\documentclass[reqno,12pt,letterpaper]{amsart}
\usepackage[proof]{sdmacros}

\usepackage{overpic}

\def\Mext{M_{\mathrm{ext}}}

\title[Improved fractal Weyl bounds]%
{Improved fractal Weyl bounds\\
for hyperbolic manifolds}
\author[Semyon Dyatlov, David Borthwick, and Tobias Weich]%
{Semyon Dyatlov\vskip.025in
with an appendix by David Borthwick, Semyon Dyatlov, and Tobias Weich}
\email{dyatlov@math.mit.edu}
\address{Department of Mathematics, Massachusetts Institute of Technology,
77 Massachusetts Ave, Cambridge, MA 02139}
\email{davidb@mathcs.emory.edu}
\address{Department of Mathematics and Computer Science, Emory University
Atlanta, GA 30322}
\email{weich@math.upb.de}
\address{Fakult\"at f\"ur Elektrotechnik, Informatik und Mathematik,
Institut f\"ur Mathematik,
Warburger Str. 100,
33098 Paderborn, Germany}

\begin{document}

\begin{abstract}
We give a new fractal Weyl upper bound for resonances of convex co-compact hyperbolic
manifolds in terms of the dimension $n$ of the manifold and
the dimension $\delta$ of its limit set. More precisely, we show
that as $R\to\infty$, the number of resonances in the box $[R,R+1]+i[-\beta,0]$
is $\mathcal O(R^{m(\beta,\delta)+})$, where the exponent
$m(\beta,\delta)=\min(2\delta+2\beta+1-n,\delta)$
changes its behavior at $\beta={n-1\over 2}-{\delta\over 2}$.
In the case $\delta<{n-1\over 2}$, we also give an improved resolvent upper bound
in the standard resonance free strip $\{\Im\lambda > \delta-{n-1\over 2}\}$.
Both results use the fractal uncertainty principle point of view
recently introduced in~\cite{hgap}.
The appendix presents numerical evidence for the Weyl upper bound.
\end{abstract}

\maketitle

\addtocounter{section}{1}
\addcontentsline{toc}{section}{1. Introduction}

In this paper we study asymptotics of scattering resonances
of convex co-compact hyperbolic quotients $(M,g)=\Gamma\backslash\mathbb H^n$.
Resonances are complex numbers which replace eigenvalues as discrete spectral
data of the Laplacian for non-compact manifolds~-- see for instance~\cite{BorthwickBook,dizzy}.
They are defined as poles of the scattering resolvent
\begin{equation}
  \label{e:scat-res}
\mathcal R(\lambda)=\Big(-\Delta_g-{(n-1)^2\over 4}-\lambda^2\Big)^{-1}:L^2_{\comp}(M)\to H^2_{\loc}(M),\quad
\lambda\in\mathbb C,
\end{equation}
which is the meromorphic continuation of the $L^2$ resolvent from the upper half-plane~-- see~\S\ref{s:vasy}.
Resonances correspond to zeroes of the Selberg zeta function~\cite[(3.1)]{GLZ}
\begin{equation}
  \label{e:selb-zeta}
Z_\Gamma(s)=\prod_{\gamma}\prod_{\alpha\in\mathbb N_0^{n-1}}
\Big(1-e^{-i\langle\theta(\gamma),\alpha\rangle}e^{-(s+|\alpha|)\ell(\gamma)}\Big),\quad
s={n-1\over 2}-i\lambda,
\end{equation}
where $\gamma$ varies in the set of primitive closed geodesics on $M$, $\ell(\gamma)$ is its period, and $\theta(\gamma)$ its holonomy spectrum~-- see~\cite{GLZ} for details.

Our main result is a bound on the number of resonances in strips, using the quantity
$$
\mathcal N(R,\beta)=\#\{\lambda\text{ resonance},\ \Re\lambda\in [R,R+1],\ \Im\lambda\geq -\beta\},\quad
R,\beta>0.
$$
\begin{theo}
  \label{t:ifwl}
Let $\delta\in [0,n-1]$ be the dimension of the limit set of $\Gamma$,
see e.g.~\cite[(5.2)]{hgap}. Then
for each $\beta\geq 0,\varepsilon>0$, there exists a constant $C$ such that
\begin{align}
  \label{e:ifwl}
\mathcal N(R,\beta)&\leq CR^{m(\beta,\delta)+\varepsilon},\quad R\to \infty;
\\
  \label{e:ifwl-exponent}
m(\beta,\delta)&:=\min(2\delta+2\beta+1-n,\delta).
\end{align}
Here resonances are counted with multiplicities, see~\eqref{e:multiplicities}.
\end{theo}
See Figure~\ref{f:basic-graphs}(a),(b). In the Appendix, we compare this upper bound
with numerically computed resonance data for several examples of hyperbolic surfaces.

The bound~\eqref{e:ifwl} is related to several previous results on distribution
of resonances (see~\cite{Nonnenmacher} for a more broad overview of results
in open quantum chaos):
\begin{figure}
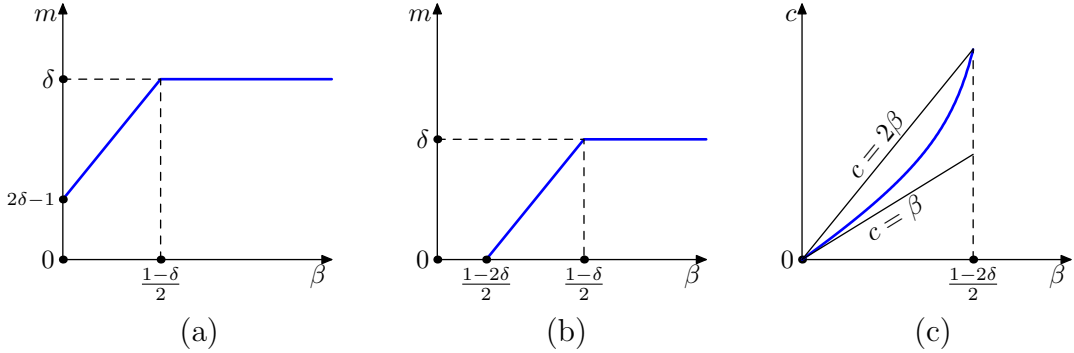

\includegraphics{ifwl.3}
\qquad
\includegraphics{ifwl.4}
\qquad
\includegraphics{ifwl.5}
\hbox to\hsize{\hss \qquad(a) \hss\hss (b) \hss\hss (c) \hss}
\caption{(a),(b) Plots of the exponent $m(\beta,\delta)$ in the Weyl bound~\eqref{e:ifwl},
for $n=2$ and (a) $\delta=0.6$ (b) $\delta=0.4$.
(c) Plot of the exponent $c(\beta,\delta)$ in the resolvent bound~\eqref{e:resolvator},
for $n=2$ and $\delta=0.15$. The straight lines are the previous resolvent bound $c=2\beta$ of~\cite{hgap}
and the lower bound $c=\beta$
of~\cite{lbnc}.}
\label{f:basic-graphs}
\end{figure}

\noindent 1. The bound
\begin{equation}
  \label{e:standard-fwl}
\mathcal N(R,\beta)\leq CR^\delta,\quad
R\to\infty
\end{equation}
was proved by Guillop\'e--Lin--Zworski~\cite{GLZ} for convex co-compact
Schottky quotients, including all convex co-compact hyperbolic surfaces.
(See also the earlier work of Zworski~\cite{ZwInventiones} in
the case of surfaces.)
Datchev--Dyatlov~\cite{fwl} proved~\eqref{e:standard-fwl}
 for all convex co-compact hyperbolic
quotients and a wider class of asymptotically hyperbolic
manifolds with hyperbolic trapped sets, using the methods
developed by Sj\"ostrand~\cite{SjostrandFWL}
and Sj\"ostrand--Zworski~\cite{SjostrandZworskiFWL}
in the case of Euclidean infinite ends.
Note that in contrast with~\eqref{e:ifwl}, the bound~\eqref{e:standard-fwl}
does not lose an $\varepsilon$ in the exponent.

\noindent 2. The standard Patterson--Sullivan spectral gap~\cite{Patterson3,Sullivan}
states that for $\delta<{n-1\over 2}$, there are no resonances in $\{\Im\lambda>\delta-{n-1\over 2}\}$, that is
$$
\mathcal N(R,\beta)=0,\quad
\beta<{n-1\over 2}-\delta.
$$
This is in agreement with the fact that $m(\beta,\delta)<0$ when $\beta<{n-1\over 2}-\delta$.
Essential gaps of larger size (depending in a complicated way on the quotient) have been
obtained by Naud~\cite{NaudGap}, Stoyanov~\cite{Stoyanov1}, and Dyatlov--Zahl~\cite{hgap}.

\noindent 3. In~\cite{Jakobson-Naud2}, Jakobson and Naud have conjectured an essential
gap of size ${n-1-\delta\over 2}$:
$$
\mathcal N(R,\beta)=0,\quad
\beta<{n-1-\delta\over 2},\quad
R\gg 1.
$$
While numerical evidence does not seem to confirm this conjecture,
it does show that the set of resonances becomes more dense
near the line $\Im\lambda=-{n-1-\delta\over 2}$~-- see the works of Borthwick~\cite[\S\S7,8]{BorthwickNum},
Borthwick--Weich~\cite[\S5.3]{Borthwick-Weich}, and the Appendix.
This is a special case of concentration of imaginary parts of resonances
near the pressure ${1\over 2}P(1)$ for open chaotic systems,
first discovered numerically by Lu--Sridhar--Zworski~\cite[Figure~2]{LSZ} for semiclassical zeta functions
on multi-disk scatterers and later observed in microwave experiments by Barkhofen \emph{et al.}~\cite[Figure~4]{ZworskiPRL}.
In the setting of open quantum maps, such concentration was observed numerically
by Shepelyanski~\cite[Figures~4 and~5]{Shepelyanski}
and Novaes~\cite{Novaes-Extra};
the recent work of Dyatlov--Jin~\cite{oqm} proves an analog of Theorem~\ref{t:ifwl}
for quantum open baker's maps.
Our exponent~\eqref{e:ifwl-exponent} is in agreement with these observations, since
it changes behavior at $\beta={n-1-\delta\over 2}$.

\noindent 4. In~\cite{NaudCount}, Naud obtained an improved Weyl upper bound
in dimension $n=2$,
$$
\mathcal N(R,\beta) \leq CR^{m'(\beta,\delta)},\quad
R\to\infty,
$$
where $m'(\beta,\delta)$ is some function satisfying
$$
m'(\beta,\delta)<\delta\quad\text{for }\beta<{1-\delta\over 2}.
$$
This result was extended to uniform bounds for congruence subgroups of arithmetic groups
by Jakobson--Naud~\cite{Jakobson-Naud3}. These bounds make essential use of total
discontinuity of the limit set, apply to surfaces only, and depend
on the choice of a particular Schottky representation of~$\Gamma$;
we also note that unlike~\eqref{e:ifwl-exponent},
$m'(\beta,\delta)$ is positive at the Patterson--Sullivan gap
$\beta={1\over 2}-\delta$.
The exponent in Theorem~\ref{t:ifwl} is always smaller than the ones
obtained in~\cite{NaudCount,Jakobson-Naud3}~-- see~\eqref{e:ifwl-ecponent-press}
and~\eqref{e:fn-relation}.

\noindent 5. We finally discuss known lower bounds on the number of resonances in strips.
Guillope--Zworski~\cite{GuillopeZworskiGAFA} showed that for $n=2$, the number
of resonances in $[0,R]+i[-\beta,0]$ cannot be $\mathcal O(R^{1-1/\beta})$, for $\beta>2$. A similar
result for higher dimensional manifolds was proved by Perry~\cite{PerryIMRN}.
Jakobson--Naud~\cite{Jakobson-Naud2}
proved that there are infinitely many resonances in $\{\Im\lambda>-\beta\}$,
for $\beta=(2\delta^2-\delta+1)/2$ for surfaces
and $\beta={3\over 4}-{\delta\over 2}$ for the special class of arithmetic surfaces with $\delta>{1\over 2}$.
Neither of these bounds matches~\eqref{e:ifwl}, since they give no information
for $\beta<{1-\delta\over 2}$ and the exponents of $R$ in the lower bounds are much smaller than $\delta$.
However, numerical computations indicate
that the bound~\eqref{e:ifwl} is saturated at least when $\beta>{n-1-\delta\over 2}$~--
see the Appendix. See also~\cite{ZworskiPRE}
for experimental data in the related case of many-disk scattering.

\smallsection{Outline of the proof of Theorem~\ref{t:ifwl}}
Theorem~\ref{t:ifwl} is proved in~\S\ref{s:proofs}; we give
an informal outline of the proof here.
We use the semiclassically rescaled spectral parameter
$\omega=h\lambda$, putting $h:=R^{-1}\ll 1$.

Assume first that $\lambda=\omega/h$ is a resonance, then there exists a
\emph{resonant state}
$$
u\in C^\infty(M),\quad
\Big(-h^2\Delta_g-{h^2(n-1)^2\over 4}-\omega^2\Big)u=0,\quad
u\text{ outgoing}.
$$
The outgoing condition can be formulated in terms of asymptotics of $u$
at the infinite ends of $M$. We use the recent approach due to Vasy~\cite{Vasy-AH1,Vasy-AH2}
(as reviewed in~\S\ref{s:vasy}; see~\cite[Chapter~5]{dizzy}
and~\cite{v4d} for expository treatments) which multiplies $u$ by a power of the boundary
defining function at conformal infinity and extends it
past the boundary of the even compactification of $M$, to obtain a smooth function
on a compact manifold without boundary $\Mext\supset M$.
The semiclassical scattering resolvent $R_h(\omega)=h^{-2}R(\omega/h)$ is expressed via the inverse of a family
of Fredholm operators (denoted $\mathcal P_h(\omega)$ in this paper) on a Sobolev space on $\Mext$. We denote by $\|u\|$ the norm of its extension to $\Mext$ in this Sobolev space.
We reduce the analysis to a compact region inside the original manifold $M$,
essentially treating the construction of~\cite{Vasy-AH1,Vasy-AH2} as a black box.

Let $\Gamma_\pm\subset T^*M$ be the incoming/outgoing tails and $K=\Gamma_+\cap\Gamma_-$
the trapped set, see~\S\ref{s:geometry}. It follows immediately from~\cite{Vasy-AH1,Vasy-AH2}
that $u$ is microlocally concentrated on $\Gamma_+\cap \{|\xi|_g=1\}$; in particular,
for each $h$-independent symbol $a(x,\xi)\in C_0^\infty(T^*M)$
and $\Op_h(a)$ the corresponding semiclassical pseudodifferential operator
(see~\cite{e-z}), we have
$$
\supp a\cap \Gamma_+\cap \{|\xi|_g=1\}=\emptyset\ \Longrightarrow\ \|\Op_h(a)u\|=\mathcal O(h^\infty)\|u\|.
$$
It was shown in~\cite[\S4.3]{hgap} (modulo localization to the cosphere bundle $\{|\xi|_g=1\}$,
which is proved in Lemma~\ref{l:elliptic-parametrix}) that $u$
is in fact microlocalized $h^\rho$ close to $\Gamma_+\cap \{|\xi|_g=1\}$,
for any $\rho<1$: namely there exists
$$
\chi_+(x,\xi;h)\in C_0^\infty(T^*M),\quad
\supp\chi_+\subset Ch^\rho\text{-neighborhood of }
\Gamma_+\cap \{|\xi|_g=1\},
$$
such that (assuming for simplicity that $\Op_h^{L_u}(1)$ is the identity operator;
see the next paragraph for the notation $\Op_h^{L_u}$)
\begin{equation}
  \label{e:chi+con}
u=\Op_h^{L_u}(\chi_+)u+\mathcal O(h^\infty)\|u\|\quad\text{microlocally near }K.
\end{equation}
In practice, we will take $\rho$ very close to 1.
The derivatives of the symbol $\chi_+$ grow like $h^{-\rho}$, therefore
it cannot be quantized using standard pseudodifferential calculus.
However, $\Gamma_+\cap \{|\xi|_g=1\}$ is foliated
by the leaves of the weak unstable foliation $L_u$ (see~\eqref{e:L-s-L-u}), and $\chi_+$
does not grow when differentiated along $L_u$. This makes it possible
to quantize $\chi_+$ using the quantization procedure $\Op_h^{L_u}$ developed in~\cite{hgap}, see~\S\ref{s:approx-inverse}.

Furthermore, \cite[\S4.3]{hgap} shows that $u$ cannot be too small on $\Gamma_-$: there exists
$$
\chi_-(x,\xi;h)\in C_0^\infty(T^*M),\quad
\supp\chi_-\subset Ch^\rho\text{-neighborhood of }\Gamma_-,
$$
such that (modulo an arbitrarily small loss in the power of $h$)
\begin{equation}
  \label{e:chi-con}
\|u\|\leq Ch^{\rho\Im\omega/h}\|\Op_h^{L_s}(\chi_-)u\|.  
\end{equation}
Here we again use the calculus of~\cite[\S3]{hgap}, this time associated
to the weak stable foliation $L_s$. Together, \eqref{e:chi+con} and~\eqref{e:chi-con}
give
\begin{equation}
  \label{e:total-con}
\|u\|\leq Ch^{\rho\Im\omega/h}\|\Op_h^{L_s}(\chi_-)\Op_h^{L_u}(\chi_+)u\|.
\end{equation}
In~\cite{hgap}, an operator norm bound on the product $\Op_h^{L_s}(\chi_-)\Op_h^{L_u}(\chi_+)$
(called the fractal uncertainty principle) was used to
show an essential spectral gap. In the present paper, we give a
stronger version of~\eqref{e:total-con}, Proposition~\ref{l:main-parametrix}, which
constructs a smoothing operator
$$
\mathcal A(\omega)=\mathcal J(\omega)\Op_h^{L_s}(\chi_-)\Op_h^{L_u}(\chi_+)+\mathcal O(h^\infty),\quad
\|\mathcal J(\omega)\|\leq Ch^{\rho\Im\omega/h},
$$
such that if $\lambda=\omega/h$ is a resonance, then $1-\mathcal A(\omega)$
is not invertible. Then each resonance produces a zero
of the Fredholm determinant
$$
F(\omega)=\det(1-\mathcal A(\omega)^2).
$$
By Jensen's inequality, to show~\eqref{e:ifwl}
with $m=2\delta+2\beta+1-n$ it remains to prove the Hilbert--Schmidt bound
(see Proposition~\ref{l:hs})
$$
\|\mathcal A(\omega)\|_{\HS}^2\leq Ch^{2\Im\omega/h+n-1-2\delta-\varepsilon}.
$$
The term $h^{2\Im\omega/h}$ comes from the operator norm of $\mathcal J(\omega)$,
thus it remains to show
\begin{equation}
  \label{e:banana}
\|\Op_h^{L_s}(\chi_-)\Op_h^{L_u}(\chi_+)\|_{\HS}^2\leq Ch^{n-1-2\delta-\varepsilon}.
\end{equation}
The latter estimate can be heuristically explained as follows: since both $\chi_\pm$
are bounded, the left-hand side of~\eqref{e:banana} should behave like
$h^{-n}$ times the volume in $T^*M$ of the set $\supp\chi_-\cap\supp\chi_+$.
Locally near any point in $K\cap\{|\xi|_g=1\}$, we may view this set as the product
of (here $\Lambda_\Gamma\subset\mathbb S^{n-1}$ denotes the limit set of the group):
\begin{enumerate}
\item an $h^\rho$ sized interval in the direction transversal to the energy surface;
\item a fixed size interval in the direction of the geodesic flow;
\item an $h^\rho$ neighborhood of $\Lambda_\Gamma$ in the stable
direction, with volume $\mathcal O(h^{\rho(n-1-\delta)})$;
\item an $h^\rho$ neighborhood of $\Lambda_\Gamma$ in the unstable
direction, with volume $\mathcal O(h^{\rho(n-1-\delta)})$.
\end{enumerate}
Thus for $\rho=1$, the volume of $\supp\chi_-\cap\supp\chi_+$ is
$\mathcal O(h^{2n-1-2\delta})$, finishing the proof.

To obtain~\eqref{e:ifwl} with $m=\delta$, we argue in the same way,
but putting $\mathcal A(\omega)=\Op_h^{L_u}(\chi_+)$ and using~\eqref{e:chi+con} only.
The support of $\chi_+$ can be viewed as a product of the four sets above,
with the set~(4) replaced by a fixed size interval, thus for $\rho=1$ it has
volume $\mathcal O(h^{n-\delta})$, leading
to the Hilbert--Schmidt bound $\|\mathcal A(\omega)\|_{\HS}^2\leq Ch^{-\delta}$
and to~\eqref{e:ifwl}.

The above proof shows why the exponent $m(\beta,\delta)$ changes behavior
at $\beta={n-1-\delta\over 2}$: past this point, the growth as $h\to 0$ of
$\|\mathcal J(\omega)\|^2$ is faster than the decay of the volume
of the $h$-neighborhood of $\Gamma_-$, thus it is no longer benefical to
use~\eqref{e:chi-con}. Therefore, for $\beta<{n-1-\delta\over 2}$
we use localization on both $\Gamma_+$ and $\Gamma_-$
and for $\beta>{n-1-\delta\over 2}$, we only use localization on $\Gamma_+$.

\smallsection{Upper bounds on the resolvent}
Using the strategy of the proof of Theorem~\ref{t:ifwl} explained above,
we also obtain the following resolvent bound inside the Patterson--Sullivan gap
(see~\S\ref{s:proofs} for the proof):
\begin{theo}
  \label{t:gapbound}
Assume that $\delta<{n-1\over 2}$. Then for each $\beta\in (0,{n-1\over 2}-\delta)$, $\psi\in C_0^\infty(M)$,
there exists $C_0$ such that for all $\varepsilon>0$,
\begin{equation}
  \label{e:resolvator}
\|\psi \mathcal R(\lambda)\psi\|_{L^2\to L^2}\leq C_\varepsilon|\lambda|^{-1+c(\beta,\delta)+\varepsilon},\quad
\Re\lambda\geq C_0,\quad
\Im\lambda\in [-\beta,1],
\end{equation}
where (see Figure~\ref{f:basic-graphs}(c))
\begin{equation}
  \label{e:c-value}
c(\beta,\delta)={\beta(n-1-2\beta)\over n-1-\delta-2\beta}.
\end{equation}
\end{theo}
The estimate~\eqref{e:resolvator} with the power
$c=2\beta$ was proved in~\cite[Theorem~3 and~(5.4)]{hgap}.
On the other hand, using the recent result of Dyatlov--Waters~\cite[Theorem~1]{lbnc}
(which applies to hyperbolic ends as explained in~\cite[\S1.2]{lbnc};
the Lyapunov exponent $\lambda_{\max}$ of the Hamiltonian
flow $H_{|\xi|_g^2}$ on the sphere bundle is equal to 2), we see that~\eqref{e:resolvator}
cannot hold with $c<\beta$. The value $c(\beta,\delta)$ given in~\eqref{e:c-value}
lies between these lower and upper bounds:
$$
\beta\leq c(\beta,\delta)< 2\beta\quad\text{for }
\beta\in \Big(0,{n-1\over 2}-\delta\Big).
$$
Note that in the degenerate case $\delta=0$, we have $c(\beta,\delta)=\beta$,
that is our upper bound matches the lower bound of~\cite{lbnc}.

\section{Approximate inverses}

In this section, we review the framework for resonances
on hyperbolic manifolds used in~\cite{hgap}. We next construct an
approximate inverse to the modified
spectral family of the Laplacian, which is one
of the key components of the proof~-- see~Proposition~\ref{l:main-parametrix}.

\subsection{Geometry and dynamics}
\label{s:geometry}

Let $(M,g)$ be an $n$-dimensional convex co-compact hyperbolic manifold;
see~\cite{BorthwickBook} for the formal definition in dimension~2 and
\cite{perry} for general dimensions.
Consider the function
$$
p\in C^\infty(T^*M\setminus 0;\mathbb R),\quad
p(x,\xi)=|\xi|_g.
$$
The Hamiltonian flow
$$
e^{tH_p}:T^*M\setminus 0\to T^*M\setminus 0
$$
is the homogeneous version of the geodesic flow.
This flow is hyperbolic in the sense that
the tangent space $T(T^*M\setminus 0)$ decomposes
into the stable, unstable, flow, and dilation directions,
see~\cite[(4.3)]{hgap}. We will use the
weak stable/unstable subbundles of $T(T^*M\setminus 0)$
\begin{equation}
  \label{e:L-s-L-u}
L_s=\mathbb RH_p\oplus E_s,\quad
L_u=\mathbb RH_p\oplus E_u,
\end{equation}
see~\cite[(4.6)]{hgap}. By~\cite[Lemma~4.1]{hgap},
$L_s$ and $L_u$ are Lagrangian foliations
in the sense of~\cite[Definition~3.1]{hgap}.

As in~\cite[\S4.1.2]{hgap}, consider a function
$$
r:M\to\mathbb R;\quad
\ddot r>0\quad\text{on }\{r\geq 0\}\cap \{\dot r=0\}
$$
where dots denote derivatives with respect to the flow $H_p$
of the lift of $r$ to $T^*M\setminus 0$.
We moreover choose $r$ so that the sublevel sets
$\{r\leq R\}$ are compact for all $R$.
In fact, one may take $r:=\tilde x^{-1}-r_1$ where
$\tilde x$ is the boundary defining function
of a conformal compactification of $M$ and
$r_1>0$ is a large constant. Then in the infinite
ends of $M$, the function $r$ roughly behaves like
the exponential of distance to the compact core.

Define the incoming/outgoing
tails
$$
\Gamma_\pm=\{(x,\xi)\in T^*M\setminus 0\mid r(e^{tH_p}(x,\xi))\text{ is bounded as }t\to \mp\infty\}
$$
and the trapped set (which we assume to be nonempty)
$$
K=\Gamma_+\cap\Gamma_-\ \subset\ \{r<0\}.
$$
Then $\Gamma_+$ is foliated by the leaves of $L_u$ and
$\Gamma_-$ is foliated by the leaves of $L_s$, as follows from~\cite[(4.8) and~(4.12)]{hgap}.
The intersection $K\cap \{|\xi|_g=c\}$ is compact for any constant $c$.

\subsection{Scattering resolvent}
\label{s:vasy}

The existence of the meromorphic continuation of the resolvent $\mathcal R(\lambda)$
defined in~\eqref{e:scat-res}
was originally proved by Mazzeo--Melrose~\cite{MazzeoMelrose},
Guillarmou~\cite{GuillarmouAH}, and Guillop\'e--Zworski~\cite{GuillopeZworskiAA};
see~\cite[\S4.2]{hgap} for more references.
We use the recent approach of Vasy~\cite{Vasy-AH1,Vasy-AH2},
refering to~\cite[\S4.2]{hgap} for details and to~\cite[Chapter~5]{dizzy},
\cite{v4d}
for expository treatments.%
\footnote{The present paper uses the original approach of~\cite{Vasy-AH1,Vasy-AH2}
featuring complex absorbing operators on a manifold without boundary. The presentation in~\cite{dizzy,v4d}
instead does analysis on a manifold with boundary. Since the differences between
these constructions lie beyond the infinity of the original
manifold $M$, either could be used in our proofs.}
This approach relies on semiclassical analysis; we refer the reader to~\cite{e-z}
and~\cite[Appendix~E]{dizzy} for an introduction
to this subject and to~\cite[\S2]{hgap} for the notation used here.

Consider the semiclassically rescaled resolvent
$$
\mathcal R_h(\omega):=h^{-2}\mathcal R(\lambda),\quad
\omega:=h\lambda\in\Omega,
$$
where we fix $\beta_0>0$ and put
\begin{equation}
  \label{e:Omega}
\Omega:=[1-2h,1+2h]+ih[-\beta_0,1].
\end{equation}
As in~\cite{Vasy-AH1,Vasy-AH2} and~\cite[\S4.2]{hgap}, we use the semiclassical differential operator
\begin{equation}
  \label{e:vasy-operator}
\mathcal P_h(\omega)\in\Psi^2_h(\Mext);\quad
\mathcal P_h(\omega)=\psi_2\Big(-h^2\Delta_g-{h^2(n-1)^2\over 4}-\omega^2\Big)\psi_1\text{ on }M,
\end{equation}
where $\Mext$ is a compact $n$-dimensional manifold without boundary
containing $M$ as an open subset and $\psi_1,\psi_2\in C^\infty(M)$ are certain nonvanishing
functions depending on $h,\omega$ and satisfying
\begin{equation}
  \label{e:conjugators}
\psi_1=\psi_2=1\quad\text{near }\{r\leq r_0\},
\end{equation}
where $r_0>0$ can be fixed arbitrarily large; note that this implies
\begin{equation}
  \label{e:same-symbol}
\sigma_h(\mathcal P_h(\omega))=p^2-\omega^2\quad\text{near }\{r\leq r_0\}.
\end{equation}
Then (see for instance~\cite[Theorem~4.3]{Vasy-AH2})
$\mathcal P_h(\omega)$ is a family of Fredholm operators $\mathcal X\to\mathcal Y$ depending
holomorphically on $\omega\in\Omega$, where
$$
\mathcal X=\{u\in H^s_h(\Mext)\mid \mathcal P_h(1)u\in H^{s-1}_h(\Mext)\},\quad
\mathcal Y:=H^{s-1}_h(\Mext),
$$
$s>{1\over 2}+\beta_0$ is fixed,
and the $h$-dependent norm on $\mathcal X$ is defined as follows:
$$
\|u\|_{\mathcal X}^2=\|u\|^2_{H^s_h(\Mext)}+\|\mathcal P_h(1)u\|^2_{H^{s-1}_h(\Mext)}.
$$
By construction of the operator $\mathcal P_h(\omega)$,
we have $\partial_\omega\mathcal P_h(\omega)\in\Psi^1_h(\Mext)$, implying that
$\mathcal P_h(\omega)-\mathcal P_h(1)\in h\Psi^1_h(\Mext)$
for $\omega\in\Omega$.
Therefore $\mathcal P_h(\omega)$ is bounded $\mathcal X\to\mathcal Y$
uniformly in $h$. Moreover, for each $u\in\mathcal X$, we have
\begin{equation}
  \label{e:Y-Y}
\|u\|_{\mathcal X}\ \leq\
C\|u\|_{H^s_h(\Mext)}+C\|\mathcal P_h(\omega)u\|_{\mathcal Y}\ \leq\
C\|u\|_{H^{s+1}_h(\Mext)}.
\end{equation}
The inverse $\mathcal P_h(\omega)^{-1}:\mathcal Y\to\mathcal X$
is meromorphic in $\omega\in\Omega$ (see for instance~\cite[Theorem~4.7]{Vasy-AH2})
and the rescaled scattering resolvent $\mathcal R_h(\omega)$ can be expressed via this inverse
(see for instance~\cite[(5.2)]{Vasy-AH2}). Therefore, Theorem~\ref{t:ifwl}
follows from an upper bound on the number of poles of $\mathcal P_h(\omega)^{-1}$.

\subsection{Approximate inverse statement}
  \label{s:approx-inverse}

Our proofs rely on semiclassical analysis; we refer the reader to~\cite{e-z} for a
comprehensive introduction and to~\cite[\S2]{hgap} for the notation used here. In particular
we use
\begin{itemize}
\item the classical symbol classes $S^k(T^*M)$, $S^k_h(T^*M)$ and the corresponding
class of pseudodifferential operators $\Psi^k_h(M)$;
\item the principal symbol map $\sigma_h:\Psi^k_h(M)\to S^k(T^*M)$;
\item the wavefront set $\WFh(A)\subset \overline T^*M$
and the elliptic set $\Ell_h(A)\subset \overline T^*M$
of $A\in\Psi^k_h(M)$ where $\overline T^*M$ is the fiber-radially compactified
cotangent bundle;
\item the class $\Psi^{\comp}_h(M)\subset\bigcap_k\Psi^k_h(M)$
of compactly supported and compactly microlocalized pseudodifferential operators.
\end{itemize}
We will moreover use the semiclassical calculus associated to a Lagrangian
foliation developed in~\cite[\S3]{hgap}. This calculus makes it possible to
quantize $h$-dependent symbols $a\in C_0^\infty(U)$ which satisfy~\cite[Definition~3.2]{hgap}
\begin{equation}
  \label{e:funny-symbols}
\sup_{x,\xi}|Y_1\dots Y_mZ_1\dots Z_k a(x,\xi;h)|\leq Ch^{-\rho k},
\end{equation}
for each vector fields $Y_1,\dots,Y_m,Z_1,\dots,Z_k$ on $U$
such that $Y_1,\dots,Y_m\in C^\infty(U;L)$.
Here $\rho\in [0,1)$, $U\subset T^*M$ is an open subset, and $L$ is a Lagrangian foliation on $U$.

The class of symbols satisfying~\eqref{e:funny-symbols} is denoted by
$S^{\comp}_{L,\rho}(U)$, and the resulting quantization procedure, by~\cite[(3.11)]{hgap}
$$
a\in S^{\comp}_{L,\rho}(U)\ \mapsto\ \Op_h^L(a):\mathcal D'(M)\to C_0^\infty(M).
$$
We denote the corresponding class of operators by~$\Psi^{\comp}_{h,L,\rho}(U)$.
By~\cite[Lemma~3.12]{hgap}, each $A\in\Psi^{\comp}_{h,L,\rho}(U)$ is pseudolocal
and compactly microlocalized; that is,
the wavefront set $\WF'_h(A)$ is a compact subset
of the diagonal of $T^*M$. Therefore, $A$ is bounded uniformly in $h$
as an operator $H^{-N}_{h}(\Mext)\to H^N_{h}(\Mext)$ for all $N$.

For symbols $a\in C_0^\infty(U)$
which belong to the class $S^0_h(T^*M)$
(in particular, all derivatives of $a$ are bounded uniformly in $h$),
$\Op_h^L$ gives a quantization procedure for the
class $\Psi^{\comp}_h(M)$ of standard compactly microlocalized semiclassical pseudodifferential operators.

\begin{figure}
\includegraphics{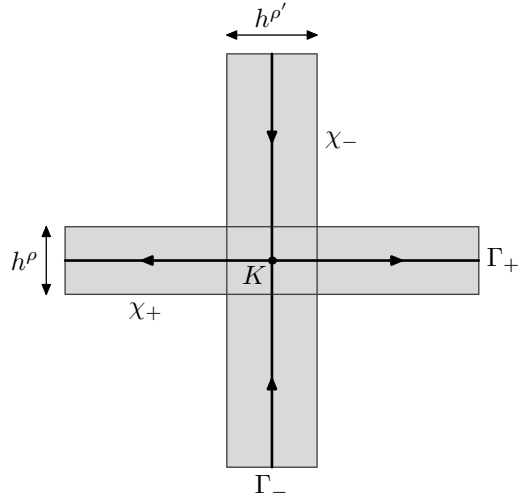}
\caption{The supports of the functions $\chi_\pm$, 
with thick lines depicting trajectories of the flow $e^{tH_p}$.
The function $\chi_+$ additionally localizes to an $h^\rho$ neighborhood
of the energy surface $\{|\xi|_g=1\}$.}
\label{f:chi-plots}
\end{figure}

We now introduce several cutoffs.
Fix $h$-independent functions
\begin{align}
  \label{e:chi-selection}
\chi\in C_0^\infty(T^*M\setminus 0;[0,1]),&\quad
\chi=1\quad\text{near }K\cap\{|\xi|_g=1\};\\
\widetilde\chi\in C_0^\infty(\mathbb R;[0,1]),&\quad
\widetilde\chi=1\quad\text{near }[-1,1].
\end{align}
Fix $\rho,\rho'\in [0,1)$ and define $h$-dependent
symbols $\chi_\pm\in C_0^\infty(T^*M\setminus 0;[0,1])$ by
\begin{equation}
  \label{e:the-cutoffs}
\begin{aligned}
\chi_+&=\chi(\chi\circ e^{-\rho\log(1/h)H_p})\widetilde\chi\Big({p-1\over h^\rho}\Big),\\
\chi_-&=\chi(\chi\circ e^{\rho'\log(1/h)H_p}).
\end{aligned}
\end{equation}
In practice, we will take $\rho$ very close to $1$ depending
on the value of $\varepsilon$ given in Theorem~\ref{t:ifwl}.
We will take $\rho'$ close to 1 to obtain the
improved exponent $m(\beta)=2\delta+2\beta+1-n$
and close to $0$ to recover the standard
exponent $m(\beta)=\delta$.

Near $K$, $\chi_+$ is a cutoff to an $h^\rho$ neighborhood of $\Gamma_+\cap \{|\xi|_g=1\}$ and $\chi_-$ is a cutoff to an $h^{\rho'}$ neighborhood of $\Gamma_-$~--
see~\cite[Lemma~4.3]{hgap} and Figure~\ref{f:chi-plots}. By~\cite[Lemma~4.2]{hgap}
and because $L_u$ is tangent to the level sets of $p$, we have
\begin{equation}
  \label{e:chi-classes}
\chi_+\in S^{\comp}_{L_u,\rho}(T^*M\setminus 0),\quad
\chi_-\in S^{\comp}_{L_s,\rho'}(T^*M\setminus 0).
\end{equation}
We are now ready to formulate the approximate inverse statement
for $\mathcal P_h(\omega)$ whose
proof occupies the rest of this section.
The proof of Theorem~\ref{t:ifwl} in~\S\ref{s:hs}
will combine this statement with a Hilbert--Schmidt
norm bound on the remainder (Proposition~\ref{l:hs}).
\begin{prop}
  \label{l:main-parametrix}
Fix $\rho,\rho'\in (0,1)$ and $\varepsilon_0>0$.
Then there exists $W\in\Psi^{\comp}_h(M)$ and $h$-dependent families of operators on $\Mext$ holomorphic in $\omega\in\Omega$,
\begin{align}
  \label{e:para-bounds-Z}
\mathcal Z(\omega):\mathcal Y\to\mathcal X,&\quad
\|\mathcal Z(\omega)\|_{\mathcal Y\to\mathcal X}\leq Ch^{-1-(\rho+\rho')(\beta_0+\varepsilon_0)};\\
  \label{e:para-bounds-J}
\mathcal J(\omega):\mathcal D'\to C^\infty,&\quad
\|\mathcal J(\omega)\|_{H^{-N}_h\to H^N_h}\leq C_Nh^{\rho'(h^{-1}\Im\omega-\varepsilon_0)}
\end{align}
with $\beta_0$ appearing in~\eqref{e:Omega}, such that for all $\omega\in\Omega$,
we have on $\mathcal X$
\begin{equation}
  \label{e:main-parametrix}
1=\mathcal Z(\omega)\mathcal P_h(\omega)+\mathcal J(\omega)\Op_h^{L_s}(\chi_-)
W\Op_h^{L_u}(\chi_+)+\mathcal E(\omega).
\end{equation}
Here the remainder $\mathcal E(\omega)$ is $\mathcal O(h^\infty)_{\mathcal D'\to C^\infty}$, meaning that for all $N$,
\begin{equation}
  \label{e:main-remainder}
\|\mathcal E(\omega)\|_{H^{-N}_h(\Mext)\to H^N_h(\Mext)}=\mathcal O(h^N).
\end{equation}
\end{prop}

\subsection{Reduction to the trapped set}

We start the proof of Proposition~\ref{l:main-parametrix} by
reducing the analysis to a fixed neighborhood of the trapped set.
This is done by means of two approximate inverse statements,
Lemma~\ref{l:basicpar-1} and~\ref{l:basicpar-2},
strengthening~\cite[Lemma~4.4]{hgap}. These statements rely on the
details of the construction of~\cite{Vasy-AH1,Vasy-AH2} and once they are established,
we may treat the infinity of $M$ as a black box.

The following lemma in particular implies that resonant
states (i.e. elements of the kernel of $\mathcal P_h(\omega)$),
when restricted to $\{r\leq r_0\}$, are microlocally negligible outside
any $h$-independent neighborhood of $\Gamma_+\cap \{|\xi|_g=1\}$:
\begin{lemm}
  \label{l:basicpar-1}
Assume that $A_1\in\Psi^0_h(\Mext)$, $\WFh(A_1)\subset \{r\leq r_0\}\subset \overline T^*M$,
where $r_0$ is given in~\eqref{e:conjugators},
and
\begin{equation}
  \label{e:basicpar-cond}
\WFh(A_1)\cap \Gamma_+\cap \{|\xi|_g=1\}=\emptyset.
\end{equation}
Then we have on $\mathcal X$,
\begin{equation}
  \label{e:basicpar-1}
A_1=Z_1(\omega)\mathcal P_h(\omega)+\mathcal O(h^\infty)_{\mathcal D'\to C^\infty},
\end{equation}
where $Z_1(\omega)$ is holomorphic in $\omega\in\Omega$
and $\|Z_1(\omega)\|_{\mathcal Y\to\mathcal X}\leq Ch^{-1}$.
\end{lemm}
\begin{figure}
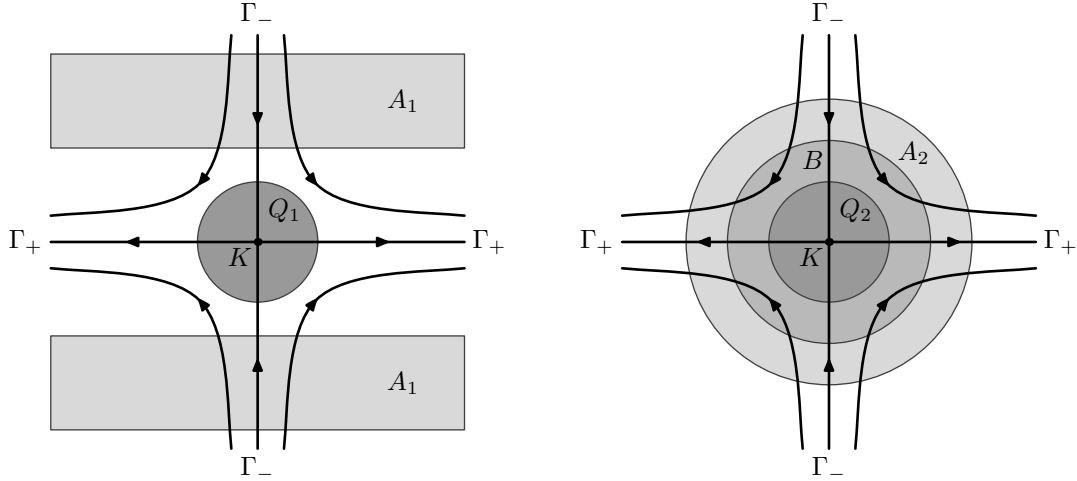

\includegraphics{ifwl.1}
\qquad
\includegraphics{ifwl.2}
\caption{An illustration of the flow $e^{tH_p}$ near $K$,
showing the wavefront
sets of the pseudodifferential operators involved
in the proofs of Lemma~\ref{l:basicpar-1} (left)
and Lemma~\ref{l:basicpar-2} (right).}
\label{f:basic}
\end{figure}
\begin{proof}
Fix a complex absorbing operator (see Figure~\ref{f:basic})
$$
\begin{aligned}
Q_1\in\Psi^{\comp}_h(M),&\quad
\sigma_h(Q_1)\geq 0;\\
K\cap \{|\xi|_g=1\}\subset\Ell_h(Q_1),&\quad
\WFh(Q_1)\subset \{r\leq r_0\}.
\end{aligned}
$$
We moreover require that $\WFh(Q_1)$ lies in a small enough
neighborhood of $K$ so that
\begin{equation}
  \label{e:basicpar-q1-dyn}
\WFh(Q_1)\cap \bigcup_{t\geq 0}e^{-tH_p}\big(\WFh(A_1)\cap \{|\xi|_g=1\}\big)=\emptyset.
\end{equation}
This is possible due to~\eqref{e:basicpar-cond}, since for each $t\geq 0$,
$e^{-tH_p}(\WFh(A_1)\cap \{|\xi|_g=1\})$ is a closed set not intersecting
$K$ and for $t$ large enough, this set lies in $\{r>r_0\}$.

The operator
\begin{equation}
  \label{e:basicpar-1-para}
\mathcal P_h(\omega)-iQ_1:\mathcal X\to\mathcal Y
\end{equation}
is invertible for $h$ small enough, and its inverse satisfies the
bound
\begin{equation}
  \label{e:basicpar-1-norm}
\|(\mathcal P_h(\omega)-iQ_1)^{-1}\|_{\mathcal Y\to\mathcal X}\leq Ch^{-1}.
\end{equation}
This follows from~\cite[Theorem~4.8]{Vasy-AH2}.
We briefly explain why this theorem applies in our case, referring
to~\cite[Theorem~5.33]{dizzy} for more details.
Consider
the rescaled Hamiltonian flow
\begin{equation}
  \label{e:Vasy-flow}
\exp(\pm t\langle \xi\rangle^{-1}H_{\tilde p}),\quad
\tilde p:=\Re\sigma_h(P_h(\omega))
\end{equation}
on the components $\Sigma_{h,\pm}\subset\overline T^*\Mext$ of the characteristic set $\{\langle\xi\rangle^{-2}\tilde p=0\}$
introduced in~\cite[\S3.4]{Vasy-AH2}.
Note that $\Sigma_{h,-}$ does not intersect $\overline T^*M$.
Then each flow line of~\eqref{e:Vasy-flow}
converges either to the radial sets $L_\pm$ or to $K$ as $t\to-\infty$;
in the latter case, this flow line lies in $\Ell_h(Q_1)$ for $-t\gg 1$.
Here we used~\cite[Lemma~4.1]{nhp}, \eqref{e:same-symbol}, and the structure of the flow~\eqref{e:Vasy-flow}
described for instance
in~\cite[Lemma~3.2]{Vasy-AH2} or~\cite[Lemma~4.4]{fwl} (see also~\cite[\S5.4]{dizzy}).
Similarly, as $t\to +\infty$ each flow line of~\eqref{e:Vasy-flow} on the characteristic set
outside of $L_\pm$ goes either to $\Ell_h(Q_1)$
or to the complex absorbing operator supported on $\Mext\setminus M$ which is part of $\mathcal P_h(\omega)$.
This means that $\mathcal P_h(\omega)-iQ_1$ satisfies the semiclassical nontrapping assumptions
described at the end of~\cite[\S3.5]{Vasy-AH2}, therefore~\cite[Theorem~4.8]{Vasy-AH2} applies.

From~\eqref{e:basicpar-q1-dyn} and~\eqref{e:same-symbol} we moreover see that each flow line of~\eqref{e:Vasy-flow}
on the characteristic set starting on $\WFh(A_1)$ does not intersect $\WFh(Q_1)$ for $t\leq 0$. Therefore,
by the semiclassically outgoing property of~\eqref{e:basicpar-1-para} (see~\cite[Theorem~4.9]{Vasy-AH2} and~\cite[Lemma~5.34]{dizzy})
we have
\begin{equation}
  \label{e:basicpar-semiout1}
A_1(\mathcal P_h(\omega)-iQ_1)^{-1}Q_1=\mathcal O(h^\infty)_{\mathcal D'\to C^\infty}.
\end{equation}
Here we used that $Q_1$ is bounded uniformly in $h$
as an operator $H^{-N}_h(\Mext)\to \mathcal Y$, for all $N$,
and the parameter $s$ in the definition of $\mathcal X$
can be chosen arbitrarily large.
Put
$$
Z_1(\omega):=A_1(\mathcal P_h(\omega)-iQ_1)^{-1}.
$$
Then the statement of the lemma follows from~\eqref{e:basicpar-1-norm}
and~\eqref{e:basicpar-1-para}, as
$$
A_1-Z_1(\omega)\mathcal P_h(\omega)=-iA_1(\mathcal P_h(\omega)-iQ_1)^{-1}Q_1.\qedhere
$$
\end{proof}
The next lemma in particular implies that each resonant state can be recovered
from its microlocal behavior in an $h$-independent neighborhood of $K\cap \{|\xi|_g=1\}$:
\begin{lemm}
  \label{l:basicpar-2}
Assume that $A_2\in\Psi^0_h(\Mext)$ is elliptic on $K\cap \{|\xi|_g=1\}$. Then
on $\mathcal X$,
\begin{equation}
  \label{e:basicpar-2}
1=Z_2(\omega)\mathcal P_h(\omega)+J_2(\omega)A_2+\mathcal O(h^\infty)_{\mathcal D'\to C^\infty}
\end{equation}
where
$Z_2(\omega),J_2(\omega)$
are holomorphic in $\omega\in\Omega$ and satisfy
$\|Z_2(\omega)\|_{\mathcal Y\to\mathcal X}\leq Ch^{-1}$, $\|J_2(\omega)\|_{H^{-N}_h\to H^N_h}\leq C_N$ for all $N$.
\end{lemm}
\begin{proof}
Fix a complex absorbing operator (see Figure~\ref{f:basic})
$$
\begin{aligned}
Q_2\in\Psi^{\comp}_h(M),&\quad
\sigma_h(Q_2)\geq 0;\\
K\cap \{|\xi|_g=1\}\subset\Ell_h(Q_2),&\quad
\WFh(Q_2)\subset \Ell_h(A_2).
\end{aligned}
$$
Take $B\in\Psi^{\comp}_h(M)$ such that
$$
\WFh(1-B)\cap\WFh(Q_2)=\emptyset,\quad
\WFh(B)\subset\Ell_h(A_2).
$$
Similarly to~\eqref{e:basicpar-1-norm}, we have for $h$ small enough,
$$
\|(\mathcal P_h(\omega)-iQ_2)^{-1}\|_{\mathcal Y\to\mathcal X}\leq Ch^{-1}.
$$
We next have
$$
(\mathcal P_h(\omega)-iQ_2)(1-B)=(1-B)\mathcal P_h(\omega)-[\mathcal P_h(\omega),B]+\mathcal O(h^\infty)_{\mathcal D'\to C^\infty}.
$$
Therefore,
$$
1=B+(\mathcal P_h(\omega)-iQ_2)^{-1}\big((1-B)\mathcal P_h(\omega)-[\mathcal P_h(\omega),B]\big)
+\mathcal O(h^\infty)_{\mathcal D'\to C^\infty}.
$$
Now, $[\mathcal P_h(\omega),B]\in h\Psi^{\comp}_h(M)$ and 
$\WFh([\mathcal P_h(\omega),B])\subset \WFh(B)\subset\Ell_h(A_2)$. Therefore,
by the elliptic parametrix construction~\cite[Proposition~E.31]{dizzy},
there exist $J',J''(\omega)\in\Psi^{\comp}_h(M)$ such that
$$
\begin{aligned}
B&=J'A_2+\mathcal O(h^\infty)_{\mathcal D'\to C^\infty},\\
[\mathcal P_h(\omega),B]&=hJ''(\omega)A_2+\mathcal O(h^\infty)_{\mathcal D'\to C^\infty}.
\end{aligned}
$$
It remains to put
$$
\begin{aligned}
Z_2(\omega)&=(\mathcal P_h(\omega)-iQ_2)^{-1}(1-B),\\
J_2(\omega)&=J'-h(\mathcal P_h(\omega)-iQ_2)^{-1}J''(\omega).\qedhere
\end{aligned}
$$
\end{proof}

\subsection{Bounded time propagation}

We next give an approximate inverse statement for operators
in classes $\Psi^{\comp}_{h,L,\rho}(T^*M\setminus 0)$, $L\in \{L_u,L_s\}$,
corresponding to propagation of singularities for a bounded time;
this is a strengthening of~\cite[Lemma~4.5]{hgap}.
The proof is an application of Egorov's theorem for the classes $\Psi^{\comp}_{h,L,\rho}$
\cite[Lemma~3.17]{hgap} together with the fundamental theorem of calculus.
This lemma is applied $\sim\log(1/h)$ times in the proof of
Propositions~\ref{l:gamma+parametrix} and~\ref{l:gamma-parametrix} below,
explaining the need for the precise norm bound~\eqref{e:J-omega-norm-bound}.
\begin{lemm}
  \label{l:finite-propagation}
Let $a,b\in S^{\comp}_{L,\rho}(T^*M\setminus 0)$ where
$L\in \{L_u,L_s\}$, $\rho\in [0,1)$, and fix $T>0$.
Assume that $|a|\leq 1$ everywhere and
\begin{equation}
  \label{e:finite-propagation-cond}
e^{-TH_p}(\supp a)\ \subset\ \{b=1\};\quad
e^{-tH_p}(\supp a)\ \subset\ W_0,\ t\in [0,T],
\end{equation}
where $W_0:=\{r\leq r_0\}\cap \{1/2\leq |\xi|_g\leq 2\}\subset T^*M\setminus 0$. Then
$$
\Op_h^L(a)=Z(\omega)\mathcal P_h(\omega)+J(\omega)\Op_h^L(b)+\mathcal O(h^\infty)_{\mathcal D'\to C^\infty}
$$
where $Z(\omega),J(\omega):\mathcal D'(\Mext)\to C^\infty(\Mext)$ are holomorphic in $\omega\in\Omega$
and for all $N$,
$$
\|Z(\omega)\|_{H^{-N}_h\to H^N_h}\leq C_Nh^{-1},\quad
\|J(\omega)\|_{H^{-N}_h\to H^N_h}\leq C_N,
$$
and for each $\varepsilon_1>0$ and $h$ small enough depending
on~$\varepsilon_1$,
\begin{equation}
  \label{e:J-omega-norm-bound}
\|J(\omega)\|_{L^2\to L^2}\leq \exp(-T\Im\omega/h)+\varepsilon_1.
\end{equation}
\end{lemm}
\begin{proof}
Let $P\in\Psi^{\comp}_h(M)\subset\Psi^{\comp}_h(\Mext)$, $P^*=P$,
be the operator constructed in~\cite[(4.22)]{hgap}; then by~\eqref{e:vasy-operator} and~\eqref{e:conjugators},
\begin{equation}
  \label{e:omega-resolved}
\begin{aligned}
\mathcal P_h(\omega)&=P^2-\omega^2\quad\text{microlocally near }W_0,\\
\sigma_h(P)&=p=|\xi|_g\quad\text{near }W_0.
\end{aligned}
\end{equation}
We have $P^2-\omega^2=(P+\omega)(P-\omega)$ and $\sigma_h(P+\omega)=p+1>0$
near $W_0$, for $\omega\in\Omega$.
By the elliptic parametrix construction~\cite[Proposition~E.31]{dizzy},
there exists a family of operators holomorphic in $\omega\in\Omega$,
\begin{equation}
  \label{e:s-operator}
S(\omega)\in\Psi^{\comp}_h(M),\quad
S(\omega)(P+\omega)=1+\mathcal O(h^\infty)\quad\text{microlocally near }W_0.
\end{equation}
By~\cite[Lemma~3.17]{hgap} and the second part of~\eqref{e:finite-propagation-cond}, there exists a family of operators
$$
A_t\in\Psi^{\comp}_{h,L,\rho}(T^*M\setminus 0),\quad
t\in [0,T],\quad
A_0=\Op^L_h(a)+\mathcal O(h^\infty)_{\mathcal D'\to C^\infty},
$$
with principal symbols $\sigma_h^L(A_t)=a\circ e^{tH_p}+\mathcal O(h^{1-\rho})_{S^{\comp}_{L,\rho}(T^*M\setminus 0)}$ and
$$
ih\partial_t A_t+[P,A_t]=\mathcal O(h^\infty)_{\mathcal D'\to C^\infty}.
$$
Let $e^{-itP/h}$ be the Schr\"odinger propagator
associated to the compactly microlocalized self-adjoint operator $P$; it is
a unitary operator on $L^2(\Mext)$
and $e^{-itP/h}-1$ is compactly microlocalized.
Then
\begin{equation}
  \label{e:basic-propagated}
A_t=e^{itP/h}\Op_h^L(a)e^{-itP/h}+\mathcal O(h^\infty)_{\mathcal D'\to C^\infty},\quad
t\in [0,T],
\end{equation}
as can be seen by differentiating $e^{-itP/h}A_te^{itP/h}$ in $t$.
Applying the fundamental theorem of calculus to
$$
\Op_h^L(a)e^{-it(P-\omega)/h}=e^{-it(P-\omega)/h}A_t+\mathcal O(h^\infty)_{\mathcal D'\to C^\infty}
$$
on the interval $[0,T]$, we get
\begin{equation}
  \label{e:baspro-1}
\Op_h^L(a)=e^{-iT(P-\omega)/h}A_T
+{i\over h}\int_0^T e^{-it(P-\omega)/h}A_t(P-\omega)\,dt
+\mathcal O(h^\infty)_{\mathcal D'\to C^\infty}.
\end{equation}
Since the wavefront set of $e^{-itP/h}$ lies in the graph of $\exp(tH_{\sigma_h(P)})$, we have by~\eqref{e:basic-propagated}
$$
\WFh(A_t)\ \subset\ \exp(-tH_{\sigma_h(P)})(\WFh(A_0))\ \subset\ W_0,\quad
t\in [0,T].
$$
By~\eqref{e:omega-resolved} and~\eqref{e:s-operator}, we have
\begin{equation}
  \label{e:baspro-2}
A_t(P-\omega)=A_tS(\omega)\mathcal P_h(\omega)+\mathcal O(h^\infty)_{\mathcal D'\to C^\infty},\quad
t\in [0,T].
\end{equation}
On the other hand, using~\cite[Lemma~3.16]{hgap} and the first part of~\eqref{e:finite-propagation-cond}
as in the proof of~\cite[Lemma~4.5]{hgap}, we write
\begin{equation}
  \label{e:baspro-3}
A_T=J'\Op_h^L(b)+\mathcal O(h^\infty)_{\mathcal D'\to C^\infty},\quad
J'\in\Psi^{\comp}_{h,L,\rho}(T^*M\setminus 0),
\end{equation}
and the principal symbol $\sigma^L_h(J')$ is equal to $a\circ e^{TH_p}+\mathcal O(h^{1-\rho})$.
Since $|a|\leq 1$ everywhere, by~\cite[Lemma~3.15]{hgap} we have for each $\varepsilon_2>0$
and $h$ small enough,
\begin{equation}
  \label{e:baspro-4}
\|J'\|_{L^2\to L^2}\leq 1+\varepsilon_2.
\end{equation}
It remains to put 
$$
\begin{aligned}
Z(\omega)&={i\over h}\int_0^T e^{-it(P-\omega)/h}A_tS(\omega)\,dt,\\
J(\omega)&=e^{-iT(P-\omega)/h}J',
\end{aligned}
$$
and use~\eqref{e:baspro-1}--\eqref{e:baspro-4} and the fact
that $\|e^{-iT(P-\omega)/h}\|_{L^2\to L^2}=\exp(-T\Im\omega/h)$.
\end{proof}
We also give a version of Lemma~\ref{l:basicpar-1} which applies
to operators in $\Psi^{\comp}_{h,L,\rho}$:
\begin{lemm}
  \label{l:basicpar-improved}
Assume that $U\subset T^*M$ is an open set, $L$ is a Lagrangian
foliation on~$U$, $\rho\in [0,1)$, and $a\in S^{\comp}_{L,\rho}(U)$ satisfies
$\supp a\subset V$, where
$$
V \subset\ U\cap \{r<r_0\}\setminus\big(\Gamma_+\cap \{|\xi|_g=1\}\big)
$$
is an $h$-independent compact subset.
Then we have on $\mathcal X$
$$
\Op_h^L(a)=Z(\omega)\mathcal P_h(\omega)+\mathcal O(h^\infty)_{\mathcal D'\to C^\infty},
$$
where $Z(\omega)$ is holomorphic in $\omega\in\Omega$
and $\|Z(\omega)\|_{\mathcal Y\to\mathcal X}\leq Ch^{-1}$.
\end{lemm}
\begin{proof}
Consider an $h$-independent function
$$
b\in C_0^\infty(U\cap \{r<r_0\}),\quad
\supp b\cap \Gamma_+\cap\{|\xi|_g=1\}=\emptyset,\quad
b=1\text{ near }V.
$$
Then by~\cite[Lemma~3.16]{hgap}, there exists $J'\in\Psi^{\comp}_{h,L,\rho}(U)$ such that
$$
\Op_h^L(a)=J'\Op_h^L(b)+\mathcal O(h^\infty)_{\mathcal D'\to C^\infty}.
$$
Now, $\Op_h^L(b)\in\Psi^{\comp}_h(M)$, therefore by Lemma~\ref{l:basicpar-1}
there exists $Z'(\omega)$ holomorphic in $\omega\in\Omega$,
with $\|Z'(\omega)\|_{\mathcal Y\to\mathcal X}\leq Ch^{-1}$ and
$$
\Op_h^L(b)=Z'(\omega)\mathcal P_h(\omega)+\mathcal O(h^\infty)_{\mathcal D'\to C^\infty}.
$$
It remains to put
$$
Z(\omega):=J'Z'(\omega).
\qedhere
$$
\end{proof}

\subsection{Long time propagation}

We now iterate Lemma~\ref{l:finite-propagation} to obtain
two statements corresponding to propagation
for time up to $\rho\log(1/h)$, which is almost twice the Ehrenfest time.

We start with the following strengthening of~\cite[(4.25)]{hgap}.
It is a refinement of Lemma~\ref{l:basicpar-1} since
the support of the symbol $\chi(1-\chi\circ e^{-tH_p})$ may come
$h^\rho$ close to $\Gamma_+$.
\begin{lemm}
  \label{l:gamma+parametrix}
Fix $\chi$ satisfying~\eqref{e:chi-selection}, $\rho\in [0,1)$, and $\varepsilon_0>0$. Then there
exists $T>0$ such that uniformly in $t\in [T,\rho\log(1/h)]$ and $\omega\in\Omega$,
$$
\Op_h^{L_u}\big(\chi(1-\chi\circ e^{-tH_p})\big)=Z_+(\omega,t)\mathcal P_h(\omega)+\mathcal O(h^\infty)_{\mathcal D'\to C^\infty}.
$$
Here $\chi(1-\chi\circ e^{-tH_p})\in S^{\comp}_{L_u,\rho}(T^*M\setminus 0)$ by~\cite[Lemma~4.2]{hgap}.
The operator $Z_+(\omega,t)$ is holomorphic in $\omega\in\Omega$ and satisfies
(with $\beta_0$ defined in~\eqref{e:Omega})
$$
\|Z_+(\omega,t)\|_{\mathcal Y\to\mathcal X}\leq Ch^{-1}\exp\big((\beta_0+\varepsilon_0)t\big),\quad
t\in [T,\rho\log(1/h)].
$$
\end{lemm}
\begin{proof}
We follow the proof of~\cite[Lemma~4.6]{hgap}. Choose $T_0>0$ such that
for all $(x,\xi)\in\{|\xi|_g=1\}$ and $t,t_1,t_2\geq T_0$, we have~\cite[(4.31),(4.32)]{hgap}:
\begin{align}
  \label{e:gamma+p1}
(x,\xi)\in\Gamma_+\cap \supp\chi\ &\Longrightarrow\ e^{-tH_p}(x,\xi)\notin\supp(1-\chi),\\
  \label{e:gamma+p2}
(x,\xi)\in e^{t_1H_p}(\supp\chi)\cap e^{-t_2H_p}(\supp\chi)\ &\Longrightarrow\ (x,\xi)\notin\supp(1-\chi).
\end{align}
Put
$$
T:=2(1+\varepsilon_0^{-1}\beta_0)T_0.
$$
Take a sequence
$$
s_{0}=0,s_{1},\dots,s_{k}=t,\quad
s_{j+1}-s_{j}\in [T/2,T].
$$
Note that $k\leq C\log(1/h)$ and for some $j$-independent $\varepsilon_1>0$,
\begin{equation}
  \label{e:flour}
\exp\big(-(s_{j+1}-s_j+T_0)\Im \omega/h\big)+\varepsilon_1<\exp\big((s_{j+1}-s_j)(\varepsilon_0-\Im\omega/h)\big).
\end{equation}
Put
$$
A^j_+:=\Op_h^{L_u}\big(\chi(1-\chi\circ e^{-s_j H_p})\big).
$$
We claim that uniformly in $j=1,\dots,k-1$,
\begin{equation}
  \label{e:gamma+inter1}
A_+^{j+1}=Z^j_+(\omega)\mathcal P_h(\omega)+J^j_+(\omega)A^j_++\mathcal O(h^\infty)_{\mathcal D'\to C^\infty}
\end{equation}
where $Z^j_+(\omega),J^j_+(\omega)$ are holomorphic in $\omega\in\Omega$ and for all $N$,
\begin{equation}
  \label{e:gamma+bounds}
\begin{aligned}
\|Z^{j}_+(\omega)\|_{\mathcal Y\to \mathcal X}&\leq Ch^{-1},\\
\|J^{j}_+(\omega)\|_{H^{-N}_h\to H^{N}_h}&\leq C_N,\\
\|J^{j}_+(\omega)\|_{L^2\to L^2}&\leq \exp\big((s_{j+1}-s_j)(\varepsilon_0-\Im\omega/h)\big).
\end{aligned}
\end{equation}
To see this, we decompose
\begin{equation}
  \label{e:chi-decomposition}
\begin{aligned}
\chi=\chi_1+\chi_2,&\quad
\chi_1,\chi_2\in C_0^\infty(T^*M\setminus 0;[0,1]),\\
\supp\chi_1\subset \{1/2<|\xi|_g<2\},&\quad
\supp\chi_2\cap\Gamma_+\cap \{|\xi|_g=1\}=\emptyset,
\end{aligned}
\end{equation}
where $\chi_1,\chi_2$ are independent of $j,h$ and for all $t\in [T_0,T_0+T]$,
$t_1,t_2\geq T_0$, we have~\cite[(4.33)--(4.35)]{hgap}:
\begin{align}
  \label{e:cookie1}
(x,\xi)\in\supp\chi_1\ &\Longrightarrow\
e^{-tH_p}(x,\xi)\notin\supp(1-\chi),\\
  \label{e:cookie2}
(x,\xi)\in e^{t_1H_p}(\supp\chi)\cap e^{-t_2H_p}(\supp\chi_1)\ &\Longrightarrow\
(x,\xi)\notin\supp(1-\chi),\\
  \label{e:cookie3}
(x,\xi)\in e^{t_1 H_p}(\supp\chi_1)\cap e^{-t_2H_p}(\supp\chi)\ &\Longrightarrow\
(x,\xi)\notin\supp(1-\chi).
\end{align}
Then for all $j=1,\dots, k-1$,
\begin{equation}
  \label{e:dynamo1}
e^{-(s_{j+1}-s_j+T_0)H_p}\big(\supp(\chi_1(1-\chi\circ e^{-s_{j+1}H_p}))\big)\ \subset\
\{\chi(1-\chi\circ e^{-s_{j}H_p})=1\}.
\end{equation}
Indeed, let $(x,\xi)\in \supp(\chi_1(1-\chi\circ e^{-s_{j+1}H_p}))$. By~\eqref{e:cookie1},
$\chi(e^{-(s_{j+1}-s_j+T_0)H_p}(x,\xi))=1$.
By~\eqref{e:cookie2} applied to $e^{-s_{j+1}H_p}(x,\xi)\in \supp(1-\chi)$,
$t_1=T_0$, $t_2=s_{j+1}$, we have
$\chi(e^{-(s_{j+1}+T_0)H_p}(x,\xi))=0$.
See Figure~\ref{f:gamma+par}.
\begin{figure}
\includegraphics{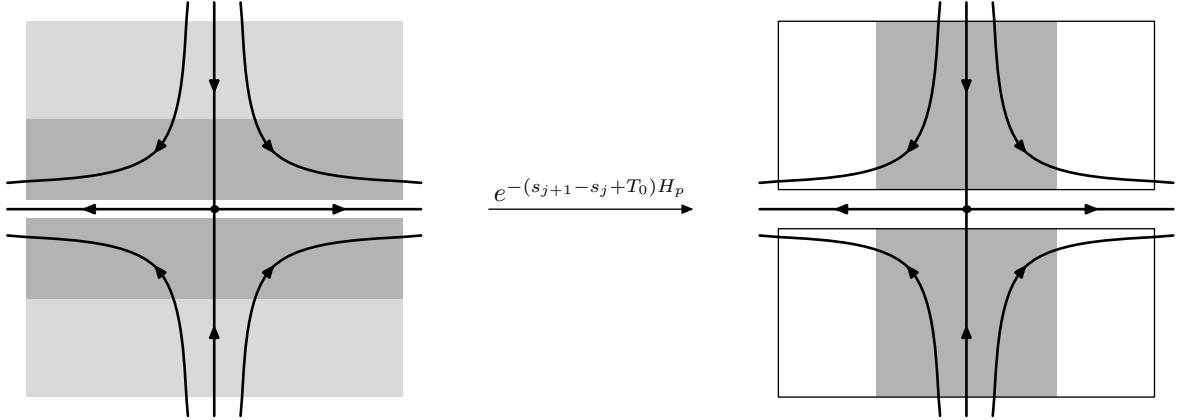}
\caption{The sets $\supp \chi_1(1-\chi\circ e^{-s_{j+1}H_p})$
(left, dark shaded),
$\supp\chi_2(1-\chi\circ e^{-s_{j+1}H_p})$
(left, light shaded),
$e^{-(s_{j+1}-s_j+T_0)H_p}\big(\supp(\chi_1(1-\chi\circ e^{-s_{j+1}H_p}))\big)$
(right, dark shaded),
and $\{\chi(1-\chi\circ e^{-s_{j}H_p})=1\}$
(right, rectangles),
illustrating~\eqref{e:dynamo1}.}
\label{f:gamma+par}
\end{figure}

To show~\eqref{e:gamma+inter1}, it now suffices to write
$$
A_+^{j+1}=\Op_h^{L_u}\big(\chi_1(1-\chi\circ e^{-s_{j+1}H_p})\big)
+\Op_h^{L_u}\big(\chi_2(1-\chi\circ e^{-s_{j+1}H_p})\big)
$$
and express the first term on the right-hand side by
Lemma~\ref{l:finite-propagation} using~\eqref{e:dynamo1}, \eqref{e:flour},
and the second term, by
Lemma~\ref{l:basicpar-improved} using~\eqref{e:chi-decomposition} and $V:=\supp\chi_2$.

By~\eqref{e:gamma+p1}, we also have
$$
\supp(\chi(1-\chi\circ e^{-s_1H_p}))\cap\Gamma_+\cap \{|\xi|_g=1\}=\emptyset.
$$
Therefore, by Lemma~\ref{l:basicpar-1}, we may write
\begin{equation}
  \label{e:gamma+inter2}
A^1_+=Z^0_+(\omega)\mathcal P_h(\omega)+\mathcal O(h^\infty)_{\mathcal D'\to C^\infty}
\end{equation}
for some $Z^0_+(\omega)$ holomorphic in $\omega\in\Omega$ with
$\|Z^0_+(\omega)\|_{\mathcal Y\to\mathcal X}\leq Ch^{-1}$.
It remains to put
$$
Z_+(\omega,t):=\sum_{j=0}^{k-1} J_+^{k-1}(\omega)\dots J_+^{j+1}(\omega)Z_+^j(\omega)
$$
and use~\eqref{e:gamma+inter1}, \eqref{e:gamma+inter2}.
\end{proof}
We next give a strengthening of~\cite[(4.26)]{hgap}.
It is a refinement of Lemma~\ref{l:basicpar-2} since the
symbol $\chi(\chi\circ e^{tH_p})$ is elliptic only
$h^\rho$ near $K$.
\begin{lemm}
  \label{l:gamma-parametrix}
Fix $\chi,\rho,\varepsilon_0$ as in Lemma~\ref{l:gamma+parametrix}. Then there exists $T>0$ such that
uniformly in $t\in [T,\rho\log(1/h)]$ and $\omega\in\Omega$,
$$
1=Z_-(\omega,t)\mathcal P_h(\omega)+J_-(\omega,t)\Op_h^{L_s}\big(\chi(\chi\circ e^{tH_p})\big)
+\mathcal O(h^\infty)_{\mathcal D'\to C^\infty}.
$$
Here $\chi(\chi\circ e^{tH_p})\in S^{\comp}_{L_s,\rho}(T^*M\setminus 0)$ by~\cite[Lemma~4.2]{hgap}.
The operators $Z_-(\omega,t),J_-(\omega,t)$
are holomorphic in $\omega\in\Omega$ and satisfy for all $N$,
$$
\begin{aligned}
\|Z_-(\omega,t)\|_{\mathcal Y\to\mathcal X}&\leq Ch^{-1}\exp\big((\beta_0+\varepsilon_0)t\big),\quad
t\in [T,\rho\log(1/h)],\\
\|J_-(\omega,t)\|_{H^{-N}_h\to H^N_h}&\leq C_N\exp\big((\varepsilon_0-\Im\omega/h)t\big),\quad
t\in [T,\rho\log(1/h)].
\end{aligned}
$$
\end{lemm}
\begin{proof}
Let $T_0,T,s_0,\dots,s_k,\chi_1,\chi_2$ be as in the proof of Lemma~\ref{l:gamma+parametrix}; put
$$
A_-^j:=\Op_h^{L_s}\big(\chi(\chi\circ e^{s_jH_p})\big).
$$
We claim that uniformly in $j=1,\dots,k-1$,
\begin{equation}
  \label{e:gamma-inter1}
A_-^j=Z_-^j(\omega)\mathcal P_h(\omega)+J_-^j(\omega)A_-^{j+1}+\mathcal O(h^\infty)_{\mathcal D'\to C^\infty},
\end{equation}
where $Z_-^j(\omega),J_-^j(\omega)$ are holomorphic in $\omega$ and satisfy the bounds~\eqref{e:gamma+bounds}.
To show this, we first claim that for all $j=1,\dots,k-1$,
\begin{equation}
  \label{e:dynamo2}
e^{-(s_{j+1}-s_j+T_0)H_p}\big(\supp(\chi_1(\chi\circ e^{s_jH_p}))\big)\ \subset\ \{\chi(\chi\circ e^{s_{j+1}H_p})=1\}.
\end{equation}
Indeed, let $(x,\xi)\in\supp(\chi_1(\chi\circ e^{s_jH_p}))$. By~\eqref{e:cookie1}, we get
$\chi(e^{-(s_{j+1}-s_j+T_0)H_p}(x,\xi))=1$. By~\eqref{e:cookie3} applied
to $e^{(s_j-T_0)H_p}(x,\xi)$, $t_1=s_j-T_0$, $t_2=T_0$, we get
$\chi(e^{(s_j-T_0)H_p}(x,\xi))=1$.
See Figure~\ref{f:gamma-par}.
Now~\eqref{e:gamma-inter1} is proved using Lemmas~\ref{l:finite-propagation} and~\ref{l:basicpar-improved}
similarly to~\eqref{e:gamma+inter1}.
\begin{figure}
\includegraphics{ifwl.9}
\caption{The sets $\supp(\chi_1(\chi\circ e^{s_jH_p}))$ (left, dark shaded),
$\supp(\chi_2(\chi\circ e^{s_jH_p}))$ (left, light shaded),
$e^{-(s_{j+1}-s_j+T_0)H_p}\big(\supp(\chi_1(\chi\circ e^{s_jH_p}))\big)$
(right, dark shaded),
and $\{\chi(\chi\circ e^{s_{j+1}H_p})=1\}$ (right, rectangle),
illustrating~\eqref{e:dynamo2}.}
\label{f:gamma-par}
\end{figure}

We next have
$$
K\cap \{|\xi|_g=1\}\ \subset\ \{\chi(\chi\circ e^{s_1H_p})=1\}.
$$
Therefore, by Lemma~\ref{l:basicpar-2}
\begin{equation}
  \label{e:gamma-inter2}
1=Z^0_-(\omega)\mathcal P_h(\omega)+J^0_-(\omega)A^1_-
+\mathcal O(h^\infty)_{\mathcal D'\to C^\infty}
\end{equation}
where $Z^0_-(\omega),J^0_-(\omega)$ are holomorphic in $\omega\in\Omega$ and satisfy
$\|Z^0_-(\omega)\|_{\mathcal Y\to\mathcal X}\leq Ch^{-1}$,
$\|J^0_-(\omega)\|_{H^{-N}_h\to H^N_h}\leq C_N$ for all $N$.

It remains to put
$$
Z_-(\omega):=\sum_{j=0}^{k-1} J^0_-(\omega)\dots J^{j-1}_-(\omega)Z^j_-(\omega),\quad
J_-(\omega):=J^0_-(\omega)\dots J^{k-1}_-(\omega)
$$
and use~\eqref{e:gamma-inter1}, \eqref{e:gamma-inter2}.
\end{proof}

\subsection{End of the proof}

We are now ready to prove Proposition~\ref{l:main-parametrix}.
By Lemma~\ref{l:gamma+parametrix} with $t:=\rho\log(1/h)$, we have
\begin{equation}
  \label{e:chi+gotcha}
\Op_h^{L_u}(\chi)=Z_+(\omega)\mathcal P_h(\omega)+\Op_h^{L_u}\big(\chi(\chi\circ e^{-\rho\log(1/h)H_p})\big)+\mathcal O(h^\infty)_{\mathcal D'\to C^\infty},
\end{equation}
where $Z_+(\omega)$ is holomorphic in $\omega\in\Omega$ and
$$
\|Z_+(\omega)\|_{\mathcal Y\to\mathcal X}\leq Ch^{-1-\rho(\beta_0+\varepsilon_0)}.
$$
We next use an elliptic estimate for symbols supported
$h^\rho$ outside of the energy surface.
Recall from~\eqref{e:the-cutoffs} that $\chi_+=\chi(\chi\circ e^{-\rho\log(1/h)H_p})\widetilde\chi\big((p-1)/h^\rho\big)$.
\begin{lemm}
  \label{l:elliptic-parametrix}
We have for $\omega\in\Omega$,
\begin{equation}
  \label{e:chi0-gotcha}
\Op_h^{L_u}\big(\chi(\chi\circ e^{-\rho\log(1/h)H_p})\big)
=Z_0(\omega)\mathcal P_h(\omega)+\Op_h^{L_u}(\chi_+)+\mathcal O(h^\infty)_{\mathcal D'\to C^\infty},
\end{equation}
with $Z_0(\omega)$ holomorphic in $\omega\in\Omega$ and
$\|Z_0(\omega)\|_{\mathcal Y\to\mathcal X}\leq Ch^{-\rho}$.
\end{lemm}
\begin{proof}
It suffices to show that for each
\begin{equation}
  \label{e:elliptic-a}
a\in S^{\comp}_{L_u,\rho}(T^*M\cap \{r<r_0\}\setminus 0),\quad
\supp a\cap \big\{|p-1|\leq h^{\rho}\big\}=\emptyset,
\end{equation}
there exists a family of operators holomorphic in $\omega\in\Omega$
$$
Z_a(\omega)\in h^{-\rho}\Psi^{\comp}_{h,L_u,\rho}(T^*M\setminus 0),\quad
\Op_h^{L_u}(a)=Z_a(\omega)\mathcal P_h(\omega)+\mathcal O(h^\infty)_{\mathcal D'\to C^\infty}.
$$
Indeed, \eqref{e:chi0-gotcha} follows by putting
$a:=\chi(\chi\circ e^{-\rho\log(1/h)H_p})-\chi_+$.

On $\supp a$, $L_u$ is tangent to level sets of $\sigma_h(P_h(\omega))=p^2-1$.
Therefore, by Darboux Theorem
(see the proof of~\cite[Lemma~3.6]{hgap})
for each $(x_0,\xi_0)\in \supp a$, there exists
a neighborhood $U_0$ of $(x_0,\xi_0)$ and a symplectomorphism
\begin{equation}
  \label{e:symplecto}
\varkappa:U_0\to T^*\mathbb R^n,\quad
\sigma_h(P_h(\omega))|_{U_0}=y_1\circ\varkappa,\quad
\varkappa_* L_u=L_0,
\end{equation}
where $L_0=\ker(dy)$ is the vertical Lagrangian foliation on $T^*\mathbb R^n$
and $y_1:\mathbb R^n\to\mathbb R$ is the first coordinate map.

By~\cite[Theorem~12.3]{e-z}, there exist Fourier integral operators
$$
B(\omega)\in I^{\comp}_h(\varkappa),\quad
B'(\omega)\in I^{\comp}_h(\varkappa^{-1})
$$
quantizing $\varkappa$ near $(x_0,\xi_0)$ in the sense of~\cite[(2.13)]{hgap} and such that
$$
\mathcal P_h(\omega)=B'(\omega)y_1B(\omega)+\mathcal O(h^\infty)\quad\text{microlocally near }(x_0,\xi_0).
$$
Applying a partition of unity to $a$, we may assume that it is supported
in a small neighborhood of $(x_0,\xi_0)$. Then by part~2 of~\cite[Lemma~3.12]{hgap},
we may write
$$
\Op_h^{L_u}(a)=B'(\omega)\Op_h(\tilde a)B(\omega)+\mathcal O(h^\infty)_{\mathcal D'\to C^\infty},\quad \tilde a\in S^{\comp}_{L_0,\rho}(T^*\mathbb R^n),
$$
where $\Op_h$ is the standard quantization procedure on $\mathbb R^n$
given by~\cite[(2.3)]{hgap}. Moreover, by~\eqref{e:elliptic-a} and~\eqref{e:symplecto}
we have
\begin{equation}
  \label{e:lalasupport}
\supp\tilde a\cap \{|y_1|\leq h^\rho\}=\emptyset.
\end{equation}
It remains to prove that there exists
$b\in h^{-\rho}S^{\comp}_{L_0,\rho}(T^*\mathbb R^n)$, $\supp b\subset\supp \tilde a$, such that
$$
\Op_h(\tilde a)=\Op_h(b)y_1+\mathcal O(h^\infty)_{L^2(\mathbb R^n)\to L^2(\mathbb R^n)}.
$$
Denote by $(y,\eta)$ the standard coordinates on $T^*\mathbb R^n$. Then by~\cite[Lemma~3.8]{hgap},
$$
\Op_h(b)y_1=\Op_h(y_1b-ih\partial_{\eta_1}b)+\mathcal O(h^\infty)_{L^2\to L^2}.
$$
Therefore we may take
$$
b\sim\sum_{j=0}^\infty b_j,\quad
b_0={\tilde a\over y_1},\quad
b_{j+1}=ih{\partial_{\eta_1}b_j\over y_1},\ j\geq 0.
$$
By induction and~\eqref{e:lalasupport}, we see that $b_j\in h^{-\rho+j(1-\rho)}S^{\comp}_{L_0,\rho}(T^*\mathbb R^n)$.
Therefore $b\in h^{-\rho}S^{\comp}_{L_0,\rho}(T^*\mathbb R^n)$, finishing the proof.
\end{proof}
Together, \eqref{e:chi+gotcha} and~\eqref{e:chi0-gotcha} give
\begin{equation}
  \label{e:chilalala}
\Op_h^{L_u}(\chi)=(Z_+(\omega)+Z_0(\omega))\mathcal P_h(\omega)
+\Op_h^{L_u}(\chi_+)+\mathcal O(h^\infty)_{\mathcal D'\to C^\infty}.
\end{equation}
Now, by Lemma~\ref{l:gamma-parametrix} with $t:=\rho'\log(1/h)$, we have
\begin{equation}
  \label{e:chi-gotcha}
1=Z_-(\omega)\mathcal P_h(\omega)+J_-(\omega)\Op_h^{L_s}(\chi_-)+\mathcal O(h^\infty)_{\mathcal D'\to C^\infty},
\end{equation}
where $Z_-(\omega),J_-(\omega)$ are holomorphic in $\omega\in\Omega$ and for all $N$,
$$
\begin{aligned}
\|Z_-(\omega)\|_{\mathcal Y\to\mathcal X}&\leq Ch^{-1-\rho'(\beta_0+\varepsilon_0)},\\
\|J_-(\omega)\|_{H^{-N}_h\to H^N_h}&\leq C_Nh^{\rho'(h^{-1}\Im\omega-\varepsilon_0)}.
\end{aligned}
$$
The final component of the proof is the following statement, reflecting the fact
that $\chi_-$ is supported very close to $\Gamma_-$, $\mathcal P_h(\omega)$
is invertible away from $\Gamma_+\cap \{|\xi|_g=1\}$ by Lemma~\ref{l:basicpar-improved},
and $\chi=1$ near $\Gamma_-\cap \Gamma_+\cap \{|\xi|_g=1\}=K\cap \{|\xi|_g=1\}$.
\begin{lemm}
We have
\begin{equation}
  \label{e:bali}
\Op_h^{L_s}(\chi_-)=Z_\chi(\omega)\mathcal P_h(\omega)+\Op_h^{L_s}(\chi_-)W\Op_h^{L_u}(\chi)
+\mathcal O(h^\infty)_{\mathcal D'\to C^\infty},
\end{equation}
for some $W\in\Psi^{\comp}_h(M)$, $Z_\chi(\omega)$ holomorphic in $\omega\in\Omega$,
and $\|Z_\chi(\omega)\|_{\mathcal Y\to\mathcal X}\leq Ch^{-1}$.
\end{lemm}
\begin{proof}
Since $\chi$ is $h$-independent, $\Op_h^{L_u}(\chi)\in\Psi^{\comp}_h(M)$.
By the elliptic parametrix construction~\cite[Proposition~E.31]{dizzy},
there exists $W\in\Psi^{\comp}_h(M)$ such that
$$
W\Op_h^{L_u}(\chi)=1\quad\text{microlocally near }\{\chi=1\}.
$$
Therefore,
$$
\Op_h^{L_s}(\chi_-)(1-W\Op_h^{L_u}(\chi))
=\Op_h^{L_s}(a)+\mathcal O(h^\infty)_{\mathcal D'\to C^\infty},
$$
for some $a\in S^{\comp}_{L_s,\rho'}(T^*M\setminus 0)$ such that
$$
\supp a\ \subset\ \supp(1-\chi)\cap \supp\chi\cap e^{-\rho'\log(1/h)H_p}(\supp\chi).
$$
Choose $T_0>0$, $\chi_1,\chi_2$ as in  the proof of Lemma~\ref{l:gamma+parametrix}. Then by~\eqref{e:cookie3} with $t_1=T_0$, $t_2=\rho'\log(1/h)$
$$
\supp a\ \subset\ V:=\supp\chi\cap e^{T_0H_p}(\supp(1-\chi_1)).
$$
Now, \eqref{e:bali} follows from Lemma~\ref{l:basicpar-improved}
once we prove that
\begin{equation}
  \label{e:i-hate-point-set-topology}
V\cap \Gamma_+\cap \{|\xi|_g=1\}=\emptyset.
\end{equation}
To show~\eqref{e:i-hate-point-set-topology}, let $(x,\xi)\in V\cap\Gamma_+\cap \{|\xi|_g=1\}$.
By~\eqref{e:gamma+p1}, $e^{-T_0H_p}(x,\xi)\notin\supp (1-\chi)$.
However, $e^{-T_0H_p}(x,\xi)\in\supp(1-\chi_1)$;
by~\eqref{e:chi-decomposition}, $\chi=\chi_1+\chi_2$
and $e^{-T_0H_p}(x,\xi)\notin\supp\chi_2$, giving a contradiction.
\end{proof}
To show Proposition~\ref{l:main-parametrix}, it now remains to put
$$
\begin{aligned}
\mathcal Z(\omega)&:=Z_-(\omega)+J_-(\omega)Z_\chi(\omega)+J_-(\omega)\Op_h^{L_s}(\chi_-)W\big(Z_+(\omega)+Z_0(\omega)\big),\\
\mathcal J(\omega)&:=J_-(\omega)
\end{aligned}
$$
and use~\eqref{e:chilalala}--\eqref{e:bali}.

\section{Hilbert--Schmidt estimates}
  \label{s:hs}

In this section, we prove a Hilbert--Schmidt norm estimate
on the operator featured in~\eqref{e:main-parametrix}. See for instance~\cite[\S B.4]{dizzy}
for an introduction to Hilbert--Schmidt operators.
See also~\cite[(5.4)]{hgap} and~\cite[Lemma~5.12]{NonnenmacherZworskiInv} for
related statements estimating the operator norm instead of the Hilbert--Schmidt norm.
\begin{prop}
  \label{l:hs}
Let $\rho,\rho'\in (0,1)$, $\varepsilon_0>0$, and
$\chi_\pm,W,\mathcal J(\omega),\mathcal E(\omega)$ be as in Proposition~\ref{l:main-parametrix}.
Then
\begin{equation}
  \label{e:aop}
\mathcal A(\omega):=\mathcal J(\omega)\Op_h^{L_s}(\chi_-)W\Op_h^{L_u}(\chi_+)
+\mathcal E(\omega),\quad
\omega\in\Omega
\end{equation}
is a Hilbert--Schmidt operator on the space $\mathcal X$, and
\begin{equation}
  \label{e:hs}
\|\mathcal A(\omega)\|^2_{\HS(\mathcal X)}\leq Ch^{-n+\rho(n-\delta)+\rho'(n-1-\delta-2\beta_0-2\varepsilon_0)}.
\end{equation}
\end{prop}
\Remark
The exponent in~\eqref{e:hs} can be heuristically explained as follows:
\begin{itemize}
\item $h^{-n}$ corresponds to restricting to frequencies $\lesssim h^{-1}$;
\item $h^{\rho(n-\delta)}$ comes from the volume of $\supp\chi_+$, which lies
inside an $h^\rho$-neighborhood of $\Gamma_+\cap \{|\xi|_g=1\}$;
\item $h^{\rho'(n-1-\delta)}$ comes from the volume of $\supp\chi_-$, which
lies inside an $h^{\rho'}$-neighborhood of $\Gamma_-$;
\item $h^{-2\rho'(\beta_0+\varepsilon_0)}$ comes from the square of the operator norm
of $\mathcal J(\omega)$, see~\eqref{e:para-bounds-J}.
\end{itemize}
To prove Proposition~\ref{l:hs}, we first note that by~\eqref{e:main-remainder}
and~\eqref{e:Y-Y}
$$
\|\mathcal E(\omega)\|_{\HS(\mathcal X)}
\leq C\|\mathcal E(\omega)\|_{\HS(H^s_h(\Mext)\to H^{s+1}_h(\Mext))}
=\mathcal O(h^\infty).
$$
By~\eqref{e:para-bounds-J} and the ideal property
of the Hilbert--Schmidt class, we then have
$$
\begin{aligned}
\|\mathcal A(\omega)\|_{\HS(\mathcal X)}
&\leq Ch^{-\rho'(\beta_0+\varepsilon_0)}\|
\Op_h^{L_s}(\chi_-)W\Op_h^{L_u}(\chi_+)\|_{\HS(\mathcal X\to L^2)}
+\mathcal O(h^\infty)\\
&\leq Ch^{-\rho'(\beta_0+\varepsilon_0)}\|
\Op_h^{L_s}(\chi_-)W\Op_h^{L_u}(\chi_+)\|_{\HS(L^2)}
+\mathcal O(h^\infty)
\end{aligned}
$$
where the last inequality follows from the fact 
$\mathcal X\subset L^2$.
Since $\Op_h^{L_s}(\chi_-)W\Op_h^{L_u}(\chi_+)$
is compactly supported on $M$,
it suffices to prove the following estimate:
\begin{equation}
  \label{e:main-hs-estimate}
\|\Op_h^{L_s}(\chi_-)W\Op_h^{L_u}(\chi_+)\|_{\HS(L^2(M))}^2
\leq Ch^{-n+\rho(n-\delta)+\rho'(n-1-\delta)}.
\end{equation}
To show~\eqref{e:main-hs-estimate}, we will follow~\cite[\S4.4]{hgap}, in particular
the proof of~\cite[Theorem~3]{hgap} there.
We start by bringing the
operator in~\eqref{e:main-hs-estimate} to a normal form. Let $\Lambda_\Gamma\subset\mathbb S^{n-1}$
be the limit set of the group $\Gamma$, $M=\Gamma\backslash\mathbb H^n$~--
see~\cite[(4.11)]{hgap}. For $\alpha>0$, denote by $\Lambda_\Gamma(\alpha)\subset \mathbb S^{n-1}$
the $\alpha$-neighborhood of $\Lambda_\Gamma$.

Lemma~\ref{l:hs-normal-form} below can be informally
explained as follows. We conjugate $\Op_h^{L_s}(\chi_-)$
by a Fourier integral operator whose underlying symplectomorphism $\varkappa_0^-$ 
`straightens out' the foliation $L_s$ (see~\eqref{e:straighten-out}),
resulting in the multiplication operator
by $\psi_-$ (times a pseudodifferential operator
which can be put into $\mathcal A_-$).
Similarly we conjugate $\Op_h^{L_u}(\chi_+)$ by a Fourier integral operator
whose underlying symplectomorphism $\varkappa_0^+$ `straightens out'
the foliation $L_u$, resulting in the multiplication operator
by $\psi_+\psi_0$. Following the above procedure for the product
$\Op_h^{L_s}(\chi_-)W\Op_h^{L_u}(\chi_+)$ also produces
a Fourier integral operator $\widetilde{\mathcal B}_\psi$ which
quantizes $\varkappa_0^-\circ(\varkappa_0^+)^{-1}$.
\begin{lemm}
  \label{l:hs-normal-form}
Let $(x_0,\xi_0)\in K\cap \{|\xi|_g=1\}$. Then there exists a neighborhood
$V\subset T^*M$ of $(x_0,\xi_0)$ such that for each
$W\in\Psi^{\comp}_h(M)$, $\WFh(W)\subset V$, we can write
$$
\begin{aligned}
\Op_h^{L_s}(\chi_-)W\Op_h^{L_u}(\chi_+)&=\mathcal A_- \widetilde{\mathcal A}\mathcal A_+
+\mathcal O(h^\infty)_{\mathcal D'\to C^\infty},\\
\widetilde{\mathcal A}&:=\psi_-(y;h)\widetilde{\mathcal B}_\psi
\psi_+(y;h)\psi_0(w;h)\widetilde\psi(hD_w)
\end{aligned}
$$
where $(w,y)$ denote coordinates on $\mathbb R^+_w\times\mathbb S^{n-1}_y$ and
\begin{itemize}
\item $\mathcal A_-:L^2(\mathbb R^+\times\mathbb S^{n-1})\to L^2(\Mext)$,
$\mathcal A_+:L^2(\Mext)\to L^2(\mathbb R^+\times\mathbb S^{n-1})$
are operators bounded uniformly in $h$ in operator norm;
\item $\psi_\pm\in C^\infty(\mathbb S^{n-1};[0,1])$ and for some constant $C_1$,
\begin{equation}
  \label{e:tight-supports}
\supp\psi_+\subset \Lambda_\Gamma(C_1h^\rho),\quad
\supp\psi_-\subset\Lambda_\Gamma(C_1h^{\rho'});
\end{equation}
\item $\psi_0\in C_0^\infty(\mathbb R^+;[0,1])$ and $\supp\psi_0\subset [1-C_1h^\rho,1+C_1h^\rho]$;
\item $\widetilde\psi\in C_0^\infty(\mathbb R;[0,1])$ is $h$-independent;
\item $\widetilde{\mathcal B}_\psi$ is the operator on $L^2(\mathbb R^+\times\mathbb S^{n-1})$
given by
$$
\widetilde{\mathcal B}_\psi v(w,y)=(2\pi h)^{1-n\over 2}\int_{\mathbb S^{n-1}}\Big|{y-y'\over 2}\Big|^{2i w/h}\psi(y,y')v(w,y')\,dy',
$$
where $|y-y'|$ denotes the Euclidean distance on the sphere $\mathbb S^{n-1}\subset\mathbb R^n$ and
$\psi\in C^\infty(\mathbb S^{n-1}\times\mathbb S^{n-1})$ is $h$-independent with
$\supp\psi\cap \{y=y'\}=\emptyset$.
\end{itemize}
\end{lemm}
\begin{proof}
We use the theory of Fourier integral operators quantizing exact
symplectomorphisms, see~\cite[\S2.2]{hgap}.
Using~\cite[Lemma~4.7]{hgap} as in~\cite[(4.57)]{hgap}, we construct
exact symplectomorphisms
$$
\varkappa_0^\pm:U\to U'_\pm,\quad
U\subset T^*M\setminus 0,\quad
U'_\pm\subset T^*(\mathbb R^+\times\mathbb S^{n-1}),
$$
where $U$ is a small neighborhood of $(x_0,\xi_0)$
and $U'_\pm$ are small neighborhoods of
$$
(1,y_0^\pm,\theta_0^\pm,\eta_0^\pm):=\varkappa_0^\pm(x_0,\xi_0).
$$
Here $(w,y,\theta,\eta)$ are the canonical coordinates on $T^*(\mathbb R^+_w\times\mathbb S^{n-1}_y)$.
The maps $\varkappa_0^\pm$ in particular straighten out the weak stable/unstable
foliations (see~\cite[(4.42)]{hgap}):
\begin{equation}
  \label{e:straighten-out}
(\varkappa^+_0)_*L_u=(\varkappa^-_0)_*L_s=L_V:=\ker(dw)\cap\ker(dy).
\end{equation}
Let $V\Subset U$ be a small neighborhood of $(x_0,\xi_0)$ and take Fourier integral operators
$$
\mathcal B_\pm\in I^{\comp}_h(\varkappa_0^\pm),\quad
\mathcal B'_\pm\in I^{\comp}_h((\varkappa_0^\pm )^{-1})
$$
which quantize $\varkappa_0^\pm$ near $V\times\varkappa_0^\pm(V)$
in the sense of~\cite[(2.13)]{hgap}:
$$
\begin{aligned}
\mathcal B'_\pm\mathcal B_\pm&=1+\mathcal O(h^\infty)\quad\text{microlocally near }V,\\
\mathcal B_\pm\mathcal B'_\pm&=1+\mathcal O(h^\infty)\quad\text{microlocally near }\varkappa_0^\pm(V).
\end{aligned}
$$
Recalling the assumption $\WFh(W)\subset V$, we now have
\begin{equation}
  \label{e:lili1}
\Op_h^{L_s}(\chi_-)W\Op_h^{L_u}(\chi_+)=\mathcal B'_-A_-BA_+\mathcal B_++\mathcal O(h^\infty)_{\mathcal D'\to C^\infty},
\end{equation}
where
$$
A_-=\mathcal B_-\Op_h^{L_s}(\chi_-)\mathcal B_-',\quad
A_+=\mathcal B_+ W\Op_h^{L_u}(\chi_+)\mathcal B_+',\quad
B=\mathcal B_-\mathcal B_+'.
$$
We have $B\in I^{\comp}_h(\widehat\varkappa^{-1})$, where
$$
\widehat\varkappa:T^*(\mathbb R^+\times\mathbb S^{n-1})\to T^*(\mathbb R^+\times\mathbb S^{n-1})
$$
is the symplectomorphism defined in~\cite[(4.45)]{hgap}, extending
$\varkappa_0^+\circ (\varkappa_0^-)^{-1}$.
By~\cite[Lemma~4.9]{hgap},
\begin{equation}
  \label{e:lili2}
B=A\widetilde{\mathcal B}_\psi+\mathcal O(h^\infty)_{\mathcal D'\to C_0^\infty},
\end{equation}
for some $A\in\Psi^{\comp}_h(\mathbb R^+\times\mathbb S^{n-1})$ and $h$-independent
$\psi\in C^\infty(\mathbb S^{n-1}\times\mathbb S^{n-1})$ such that $\supp\psi\cap \{y=y'\}=\emptyset$.

By~\eqref{e:chi-classes}, ~\eqref{e:straighten-out}, and the properties of $\Psi^{\comp}_{h,L,\rho}$ calculus
discussed in~\cite[\S3.3]{hgap},
$$
A_+\in\Psi^{\comp}_{h,L_V,\rho}(T^*(\mathbb R^+\times\mathbb S^{n-1})),\quad
A_-A\in\Psi^{\comp}_{h,L_V,\rho'}(T^*(\mathbb R^+\times\mathbb S^{n-1})).
$$
As in the discussion following~\cite[(4.59)]{hgap},
by~\eqref{e:the-cutoffs} and~\cite[Lemma~4.3 and~(4.44)]{hgap}
there exists a constant $C_1>0$ such that $A_+=\mathcal O(h^\infty)$ in the sense of~\cite[Definition~3.13]{hgap}
microlocally along each sequence
$(w_j,y_j,\theta_j,\eta_j,h_j)$ such that
$$
d(y_j,\Lambda_\Gamma)+|w_j-1|\geq C_1h_j^\rho/2
$$
Similarly, $A_-A=\mathcal O(h^\infty)$ microlocally along each sequence
such that
$$
d(y_j,\Lambda_\Gamma)\geq C_1h_j^{\rho'}/2.
$$
Using~\cite[Lemma~3.3]{hgap},
take functions $\psi_\pm(y;h),\psi_0(w;h)$ satisfying the properties in the statement
of this Lemma and such that
$$
\begin{aligned}
\supp(1-\psi_+)\cap \Lambda_\Gamma(C_1 h^\rho/2)=\emptyset,&\quad
|\partial^\alpha_y\psi_+|\leq C_\alpha h^{-\rho|\alpha|};\\
\supp(1-\psi_-)\cap \Lambda_\Gamma(C_1 h^{\rho'}/2)=\emptyset,&\quad
|\partial^\alpha_y\psi_-|\leq C_\alpha h^{-\rho'|\alpha|};\\
\supp(1-\psi_0)\cap [1-C_1h^{\rho}/2,1+C_1h^{\rho}/2]=\emptyset,&\quad
|\partial^\alpha_w \psi_0|\leq C_\alpha h^{-\rho|\alpha|}.
\end{aligned}
$$
Since $A_+$ is compactly microlocalized, there exists $R>0$ such that
$\WFh(A_+)\subset\{|\theta|<R\}$, where $\theta$ is the momentum corresponding
to $w$.
Take $h$-independent
$$
\widetilde\psi\in C_0^\infty(\mathbb R),\quad 
\widetilde\psi=1\text{ near }[-R,R].
$$
Arguing as in the proof of~\cite[(4.51)]{hgap}, we see that
\begin{align}
  \label{e:lili3}
(1-\psi_+(y;h)\psi_0(w;h)\widetilde\psi(hD_w))A_+&=\mathcal O(h^\infty)_{\mathcal D'\to C_0^\infty},\\
  \label{e:lili4}
A_-A(1-\psi_-(y;h))&=\mathcal O(h^\infty)_{\mathcal D'\to C_0^\infty}.
\end{align}
It remains to put
$$
\mathcal A_-:=\mathcal B_-'A_-A,\quad
\mathcal A_+:=A_+\mathcal B_+
$$
and use~\eqref{e:lili1}--\eqref{e:lili4}.
\end{proof}
We next estimate the operator appearing in Lemma~\ref{l:hs-normal-form}:
\begin{lemm}
  \label{l:hs-inside}
Let $\widetilde{\mathcal A}$ be the operator on $L^2(\mathbb R^+\times\mathbb S^{n-1})$ defined
in Lemma~\ref{l:hs-normal-form}. Then
$$
\|\widetilde{\mathcal A}\|^2_{\HS(L^2)}\leq Ch^{-n+\rho(n-\delta)+\rho'(n-1-\delta)}.
$$
\end{lemm}
\begin{proof}
A direct calculation shows that
$$
\widetilde{\mathcal A}v(w,y)=\int_{\mathbb R^+\times \mathbb S^{n-1}} \mathcal K(w,y,w',y')v(w',y')\,dw'dy',
$$
where the Schwartz kernel $\mathcal K$ is given by
$$
\begin{gathered}
\mathcal K(w,y,w',y')=(2\pi h)^{-{n+1\over 2}}\Big|{y-y'\over 2}\Big|^{2iw/h}\psi(y,y')\psi_-(y;h)\psi_+(y';h)\psi_0(w;h)\\
\cdot\int_{\mathbb R}e^{{i\over h}(w-w')\theta}
\widetilde\psi(\theta)\,d\theta.
\end{gathered}
$$
Using the nonsemiclassical Fourier transform
$\mathcal F(\widetilde\psi)\in \mathscr S(\mathbb R)$, we write
$$
\begin{gathered}
|\mathcal K(w,y,w',y')|\leq C h^{-{n+1\over 2}}
\psi_-(y;h)\psi_+(y';h)\psi_0(w;h)\Big|\mathcal F(\widetilde\psi)\Big({w'-w\over h}\Big)\Big|.
\end{gathered}
$$
Since $\supp \psi_0\subset [1-C_1h^\rho,1+C_1h^\rho]$, we have
$$
\int_{\mathbb R^+\times\mathbb R^+} |\mathcal K(w,y,w',y')|^2\,dwdw'\leq Ch^{\rho-n}\psi_-(y;h)^2\psi_+(y';h)^2.
$$
Now, by~\eqref{e:tight-supports} and~\cite[(1.5)]{hgap}, we have the Lebesgue measure bounds
$$
\mu_L(\supp\psi_+)\leq Ch^{\rho(n-1-\delta)},\quad
\mu_L(\supp\psi_-)\leq Ch^{\rho'(n-1-\delta)}.
$$
Therefore,
$$
\|\widetilde{\mathcal A}\|_{\HS(L^2)}^2=\int_{(\mathbb R^+\times\mathbb S^{n-1})^2} |\mathcal K(w,y,w',y')|^2
\,dwdydw'dy'\leq Ch^{-n+\rho(n-\delta)+\rho'(n-1-\delta)},
$$
finishing the proof.
\end{proof}
We now finish the proof of Proposition~\ref{l:hs}. By~\eqref{e:the-cutoffs}, we have
$$
\begin{aligned}
\WFh(\Op_h^{L_u}(\chi_+))\ &\subset\ \Gamma_+\cap \{|\xi|_g=1\}\cap \supp\chi,\\
\WFh(\Op_h^{L_s}(\chi_-))\ &\subset\ \Gamma_-\cap\supp \chi.
\end{aligned}
$$
Indeed, if $a\in C_0^\infty(T^*M)$ is $h$-independent and
$\supp a\cap \Gamma_+\cap \{|\xi|_g=1\}\cap \supp\chi=\emptyset$, then
for $h$ small enough we have $\supp a\cap\supp\chi_+=\emptyset$
and thus $\Op_h^{L_u}(\chi_+)\Op_h(a)=\mathcal O(h^\infty)_{\mathcal D'\to C^\infty}$.
The case of $\chi_-$ is handled similarly.

It follows that for $\WFh(W)\cap K\cap \{|\xi|_g=1\}=\emptyset$,
the left-hand side of~\eqref{e:main-hs-estimate}
is $\mathcal O(h^\infty)$. Combining this
with a partition of unity argument, we see that
it suffices to consider the case of $W$ satisfying the assumptions
of Lemma~\ref{l:hs-normal-form}. By Lemmas~\ref{l:hs-normal-form}
and~\ref{l:hs-inside}, we obtain~\eqref{e:main-hs-estimate}.

\section{Proof of Theorems~\ref{t:ifwl} and~\ref{t:gapbound}}
  \label{s:proofs}

We now combine Propositions~\ref{l:main-parametrix} and~\ref{l:hs} with
Jensen's inequality and Fredholm determinants 
(which are both standard tools in resonance counting bounds) to obtain
\begin{proof}[Proof of Theorem~\ref{t:ifwl}]
Let $\mathcal X,\mathcal Y$ be the Hilbert spaces
and $\mathcal P_h(\omega):\mathcal X\to\mathcal Y$ the operator
introduced in~\S\ref{s:vasy}. Fix $\beta\geq 0$, $\beta_0>\beta$, define
$\Omega$ by~\eqref{e:Omega}, and put
$$
\Omega':=[1,1+h]+ih[-\beta,1/2]\ \Subset\ \Omega.
$$
Let $m\in\mathbb R$. Putting $h:=R^{-1}$, we see that the bound on resonances
$$
\mathcal N(R,\beta)\leq CR^{m},\quad R\to \infty
$$
follows from the following bound on the poles of $\mathcal P_h(\omega)^{-1}$,
counted with multiplicities:
\begin{equation}
  \label{e:poleb-1}
\#\{\omega\text{ pole of }\mathcal P_h(\omega)^{-1},\ \omega\in\Omega'\}\leq Ch^{-m}.
\end{equation}
Here we use~\cite[Theorem~2.1]{Gohberg-Sigal} (see also~\cite[(4.3)]{nhp}) to define the multiplicity of a pole $\omega_1$ as
\begin{equation}
  \label{e:multiplicities}
{1\over 2\pi i}\tr \oint_{\omega_1} \mathcal P_h(\omega)^{-1}\partial_\omega \mathcal P_h(\omega)\,d\omega,
\end{equation}
where the integral is taken over a contour enclosing $\omega_1$, but no other poles
of $\mathcal P_h(\omega)^{-1}$. See~\cite[\S C.4]{dizzy} for an introduction to Gohberg--Sigal theory which
we use here.

Fix $\rho,\rho'\in (0,1),\varepsilon_0>0$ to be chosen later and let $\mathcal A(\omega)$ be the operator introduced in~\eqref{e:aop}.
The operator
$(1-\mathcal A(\omega)^2)^{-1}:\mathcal X\to\mathcal X$ is meromorphic in $\omega\in\Omega$ with poles of finite rank
by~\cite[Theorem~D.4]{e-z}, since $\mathcal A(\omega)^2$ is compact (as $\mathcal A(\omega)$ is a
Hilbert--Schmidt operator) and as follows from~\eqref{e:a-invertible}
below, $1-\mathcal A(\omega_0)^2$ is invertible for some $\omega_0\in\Omega$.
By Proposition~\ref{l:main-parametrix},
$$
\mathcal P_h(\omega)^{-1}=(1-\mathcal A(\omega)^2)^{-1}(1+\mathcal A(\omega))\mathcal Z(\omega):\mathcal Y\to\mathcal X.
$$
Therefore, \eqref{e:poleb-1} follows from the bound (counting poles with multiplicities; see~\cite[Theorem~C.8]{dizzy})
\begin{equation}
  \label{e:poleb-2}
\#\{\omega\text{ pole of }(1-\mathcal A(\omega)^2)^{-1},\ \omega\in\Omega'\}\leq Ch^{-m}.
\end{equation}
By Proposition~\ref{l:hs}, $\mathcal A(\omega)$ is a Hilbert--Schmidt operator on $\mathcal X$ for $\omega\in\Omega$,
therefore (see for instance~\cite[\S B.4]{dizzy}) the operator
$\mathcal A(\omega)^2:\mathcal X\to\mathcal X$ is trace class.
By~\cite[\S B.5]{dizzy} we may define the determinant
$$
F(\omega):=\det(1-\mathcal A(\omega)^2),\ \omega\in\Omega,
$$
which is a holomorphic function. By~\eqref{e:multiplicities} and since
$$
{F'(\omega)\over F(\omega)}=-\tr\big((1-\mathcal A(\omega)^2)^{-1}\partial_\omega(\mathcal A(\omega)^2)\big),
$$
we see that \eqref{e:poleb-2} follows from the following bound
(counting zeroes of $F(\omega)$ with multiplicities)
\begin{equation}
  \label{e:poleb-3}
\#\{\omega\text{ zero of }F(\omega),\quad
\omega\in\Omega'\}\leq Ch^{-m}.
\end{equation}
By~\eqref{e:hs}, we have the trace class norm bound
\begin{equation}
  \label{e:tr}
\|\mathcal A(\omega)^2\|_{\TR}\leq \|\mathcal A(\omega)\|_{\HS}^2\leq
Ch^{-n+\rho(n-\delta)+\rho'(n-1-\delta-2\beta_0-2\varepsilon_0)},\quad
\omega\in\Omega.
\end{equation}
This implies (see for instance~\cite[\S B.5]{dizzy})
\begin{equation}
  \label{e:det1}
|F(\omega)|\leq \exp(Ch^{-n+\rho(n-\delta)+\rho'(n-1-\delta-2\beta_0-2\varepsilon_0)}),\quad
\omega\in\Omega.
\end{equation}
On the other hand, we see immediately from~\eqref{e:aop}, \eqref{e:para-bounds-J},
and the fact that $\Op_h^{L_u}(\chi_+)$, $\Op_h^{L_s}(\chi_-)$, and $W$ are bounded on $\mathcal X$
uniformly in $h$ that
$$
\|\mathcal A(\omega)\|_{\mathcal X\to\mathcal X}\leq Ch^{\rho'(h^{-1}\Im\omega-\varepsilon_0)}.
$$
Taking $\varepsilon_0<1/3$, we see that for $h$ small enough,
\begin{equation}
  \label{e:a-invertible}
\|\mathcal A(\omega_0)^2\|_{\mathcal X\to\mathcal X}\leq {1\over 2},\quad
\omega_0:=1+{ih\over 3}\in\Omega'.
\end{equation}
We have
$$
\begin{aligned}
(1-\mathcal A(\omega_0)^2)^{-1}&=1+\mathcal A(\omega_0)^2(1-\mathcal A(\omega_0)^2)^{-1},\\
\|\mathcal A(\omega_0)^2(1-\mathcal A(\omega_0)^2)^{-1}\|_{\TR}&\leq \|\mathcal A(\omega_0)^2\|_{\TR}\cdot
\|(1-\mathcal A(\omega_0)^2)^{-1}\|_{\mathcal X\to\mathcal X}\\
&\leq Ch^{-n+\rho(n-\delta)+\rho'(n-1-\delta-2\beta_0-2\varepsilon_0)}.
\end{aligned}
$$
By multiplicativity of determinants, we get
\begin{equation}
  \label{e:det2}
|F(\omega_0)|^{-1}=\det\big((1-\mathcal A(\omega_0)^2)^{-1}\big)\leq
\exp(Ch^{-n+\rho(n-\delta)+\rho'(n-1-\delta-2\beta_0-2\varepsilon_0)}).
\end{equation}
By Jensen's inequality (see for instance the proof of~\cite[Theorem~2]{fwl}),
the determinant bounds
\eqref{e:det1} and~\eqref{e:det2} together imply the counting bound~\eqref{e:poleb-3}
with
$$
m=n-\rho(n-\delta)-\rho'(n-1-\delta-2\beta_0-2\varepsilon_0).
$$
To show~\eqref{e:ifwl}, it remains to choose $\rho,\rho',\beta_0,\varepsilon_0$ which
yield the following values of $m$:
\begin{itemize}
\item $m\leq 2\delta+2\beta+1-n+\varepsilon$: choose
$$
\rho=\rho'=1-\varepsilon_0,\quad
\beta_0=\beta+\varepsilon_0
$$
and take $\varepsilon_0>0$ small enough depending on $\varepsilon$;
\item $m\leq \delta+\varepsilon$: choose
$$
\rho=1-\varepsilon_0,\quad \rho'=\varepsilon_0,\quad \beta_0=\beta+\varepsilon_0
$$
and take $\varepsilon_0>0$ small enough depending on $\varepsilon$.\qedhere
\end{itemize}
\end{proof}
We finally give the proof of the resolvent bound in the Patterson--Sullivan gap:
\begin{proof}[Proof of Theorem~\ref{t:gapbound}]
As in~\cite[(4.16)]{hgap}, it suffices to show the bound
$$
\|\mathcal P_h(\omega)^{-1}\|_{\mathcal Y\to\mathcal X}\leq Ch^{-1-c(\beta,\delta)-\varepsilon},\quad
\omega\in \Omega:=[1-2h,1+2h]+ih[-\beta,1].
$$
Take $\varepsilon_0>0$ small enough to be chosen later and choose
(here the choice of $\rho'$ is explained by~\eqref{e:alpha-pos} below)
$$
\rho=1-\varepsilon_0,\quad
\rho'={\delta+\sqrt{\varepsilon_0}\over n-1-\delta-2\beta},\quad
\beta_0:=\beta.
$$
Here $\rho\in (0,1)$ for small enough $\varepsilon_0$ since
$$
\beta+\delta<{n-1\over 2}.
$$
Let $\mathcal A(\omega)$ be the operator defined in~\eqref{e:aop}.
Estimating the operator norm of this operator by its Hilbert--Schmidt
norm and using~\eqref{e:hs}, we get
$$
\|\mathcal A(\omega)\|_{\mathcal X\to\mathcal X}\leq Ch^{\alpha/2},\quad
\omega\in\Omega
$$
where
\begin{equation}
  \label{e:alpha-pos}
\alpha=-n+\rho(n-\delta)+\rho'(n-1-\delta-2\beta_0-2\varepsilon_0)=\sqrt{\varepsilon_0}+\mathcal O(\varepsilon_0)
\end{equation}
is positive for $\varepsilon_0$ small enough. Then for $h$ small enough,
\begin{equation}
  \label{e:a-nice}
\|(1-\mathcal A(\omega))^{-1}\|_{\mathcal X\to\mathcal X}\leq C.
\end{equation}
By~\eqref{e:main-parametrix}, we have
$$
\mathcal P_h(\omega)^{-1}=(1-\mathcal A(\omega))^{-1}\mathcal Z(\omega):\mathcal Y\to\mathcal X.
$$
By~\eqref{e:para-bounds-Z} and~\eqref{e:a-nice}, we get
$$
\|\mathcal P_h(\omega)^{-1}\|_{\mathcal Y\to\mathcal X}\leq Ch^{-1-\tilde c},\quad
\omega\in\Omega,
$$
where, with $c(\beta,\delta)$ given by~\eqref{e:c-value},
$$
\tilde c=(\rho+\rho')(\beta+\varepsilon_0)=c(\beta,\delta)+\mathcal O(\sqrt{\varepsilon_0}).
$$
By choosing $\varepsilon_0$ small enough depending on $\varepsilon$, we can make
$\tilde c\leq c(\beta,\delta)+\varepsilon$, finishing the proof.
\end{proof}

\appendix
\section*{Appendix: Numerical experiments\\
with David Borthwick and Tobias Weich}
\setcounter{section}{1}
\setcounter{equation}{0}

In this Appendix we compare the upper bound on the density of resonances
obtained in Theorem~\ref{t:ifwl} to numerical computations of the resonance density
for several explicit examples of convex co-compact hyperbolic surfaces. 

\subsection{Examples of hyperbolic surfaces}

Any convex co-compact hyperbolic surface can be obtained as a quotient of the
hyperbolic upper half-plane
$$
\mathbb H^2 = \SL(2,\mathbb R)/\SO(2)
$$
by a classical 
Schottky group~\cite{Button}. Such a Schottky group is a discrete subgroup 
$\Gamma\subset \SL(2,\mathbb{R})$ freely generated by $r\geq 1$
hyperbolic elements $g_1,\ldots, g_r\in \SL(2,\mathbb{R})$ which
fulfill mapping conditions on a set of 
disjoint disks
$$
D_1,\ldots D_{2r}\ \subset\ \mathbb C,
$$
with centers on $\partial \mathbb H^2$.  
In particular $g_j$ maps the interior of $D_j$ precisely to the exterior of $D_{j+r}$.  
We refer to \cite[\S15.1]{BorthwickBook} for a detailed introduction. 

In the simplest case $r=1$, the group is cyclic and the quotient surface is a hyperbolic cylinder.  
For these cases the resonances spectrum is explicitly known \cite[\S5.1]{BorthwickBook}.  
However, as $\delta = 0$ for these elementary surfaces, they are not of interest
for the improved upper bounds of Theorem~\ref{t:ifwl}.

The simplest nontrivial 
case are the surfaces with $r=2$ generators, which all have $\delta >0$. 
There exist only two topological types of these surfaces,
the three-funnel surfaces and the funneled tori (see Figure~\ref{surf_pic.fig}). The moduli spaces
of these two topological types of Schottky surfaces are in both cases three-dimensional. 
In the case of the three-funnel Schottky surfaces the parameters
of the moduli space can be chosen to be $(l_1,l_2,l_3) \in (\mathbb R^+)^3$. These
numbers coincide with the lengths of the three simple closed geodesics that bound 
the funnels.  We denote these surfaces by $X(l_1,l_2,l_3)$. 
For the funneled tori, the parameters can be chosen to be 
$(l_1,l_2,\phi) \in (\mathbb R^+)^2\times (0,\pi)$, consisting of two lengths of simple 
closed geodesics $\phi$ the angle between them (see Figure~\ref{surf_pic.fig}).
These surfaces will be denoted by $Y(l_1,l_2,\phi)$.
\begin{figure}
\begin{tabular}{ccc}
\begin{overpic}[scale=.8]{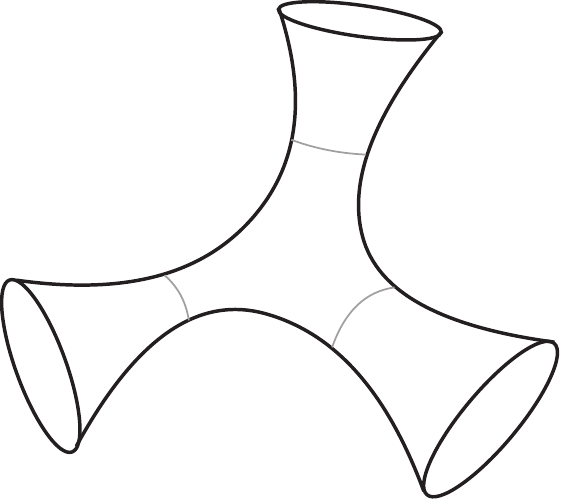}
\put(25,43){$l_1$}
\put(67,57){$l_2$}
\put(55,20){$l_3$}
\end{overpic} &&
\quad
\begin{overpic}[scale=.8]{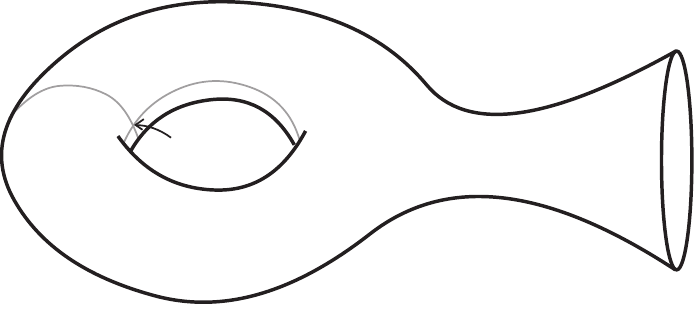}
\put(8,26){$l_1$}
\put(40,31){$l_2$}
\put(24,21){$\phi$}
\end{overpic}\\
$X(l_1,l_2,l_3)$ && $Y(l_1,l_2,\phi)$
\end{tabular}
\caption{Schottky surfaces with 2 generators: the three-funnel surfaces
and the funneled tori.}\label{surf_pic.fig}
\end{figure}
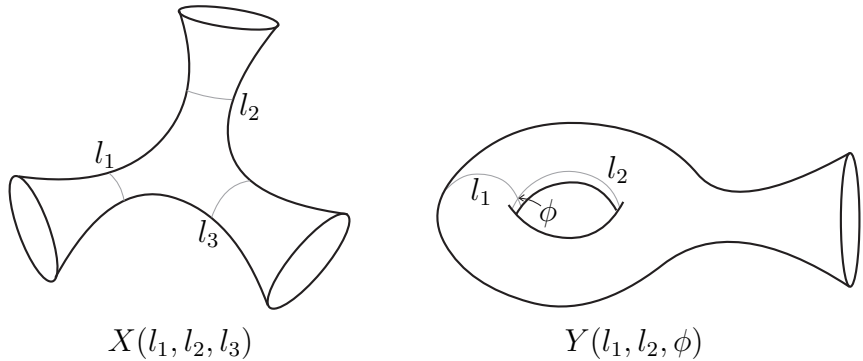

\subsection{Numerical resonance calculation via dynamical zeta functions}
\label{s:numerical-method}

\begin{figure}
\includegraphics[width=6.25in]{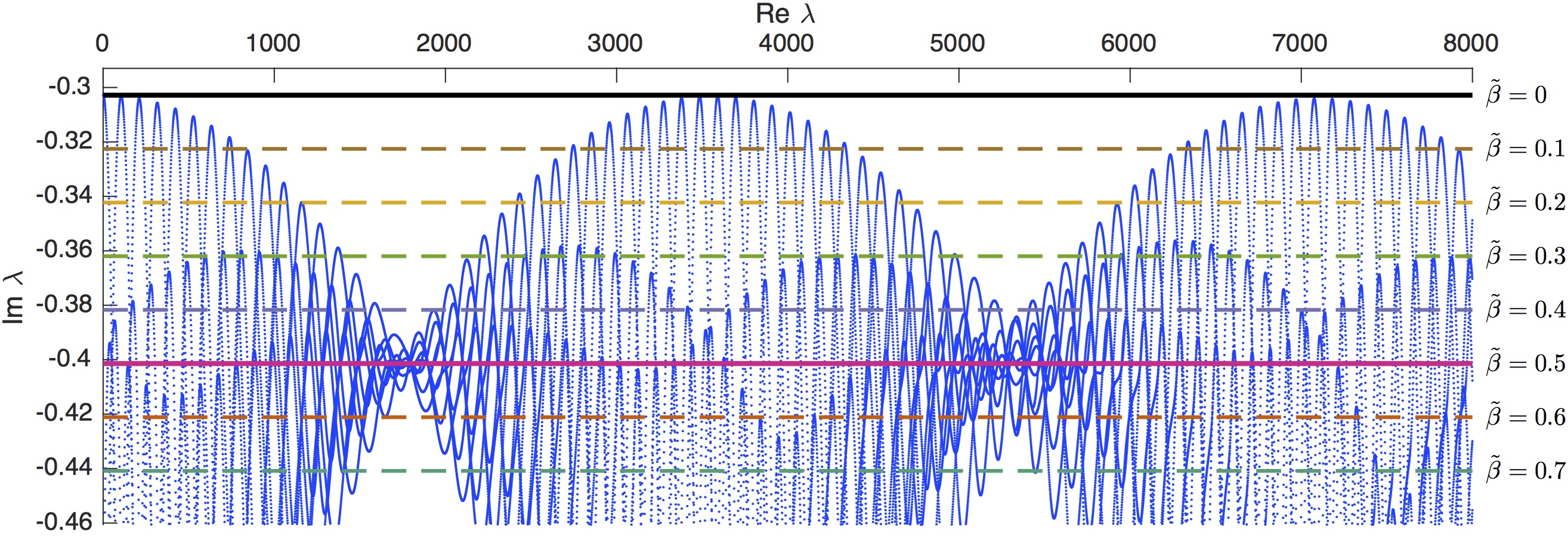}\\
\includegraphics[width=6.25in]{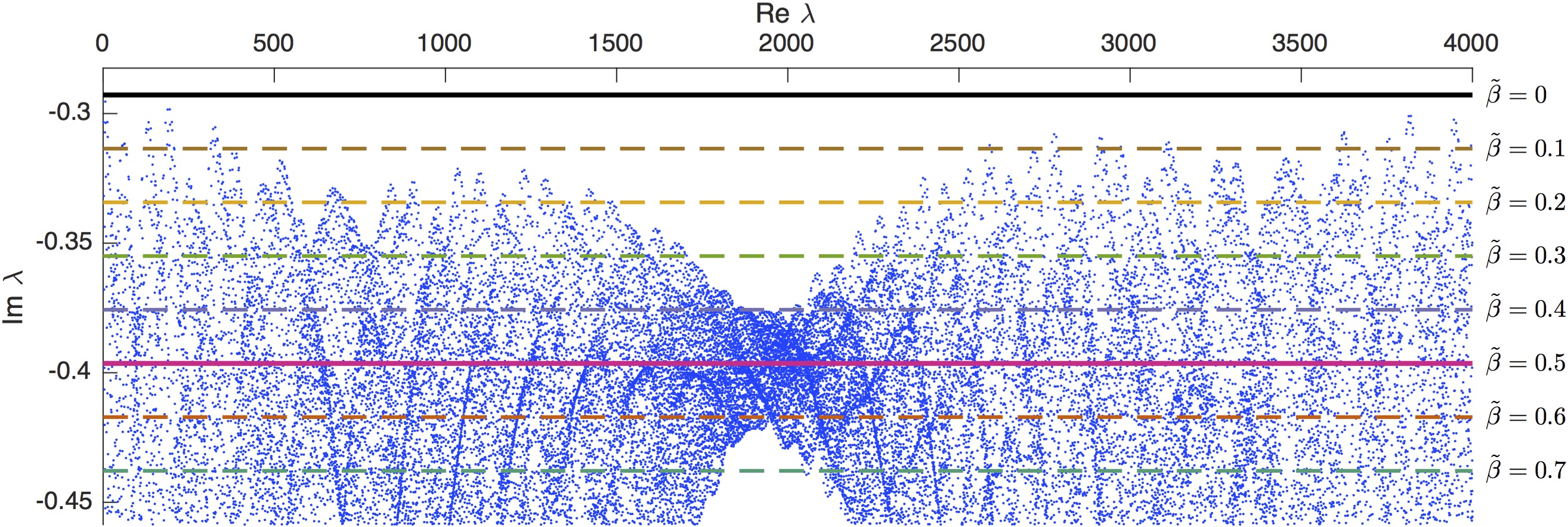}\\
\includegraphics[width=6.25in]{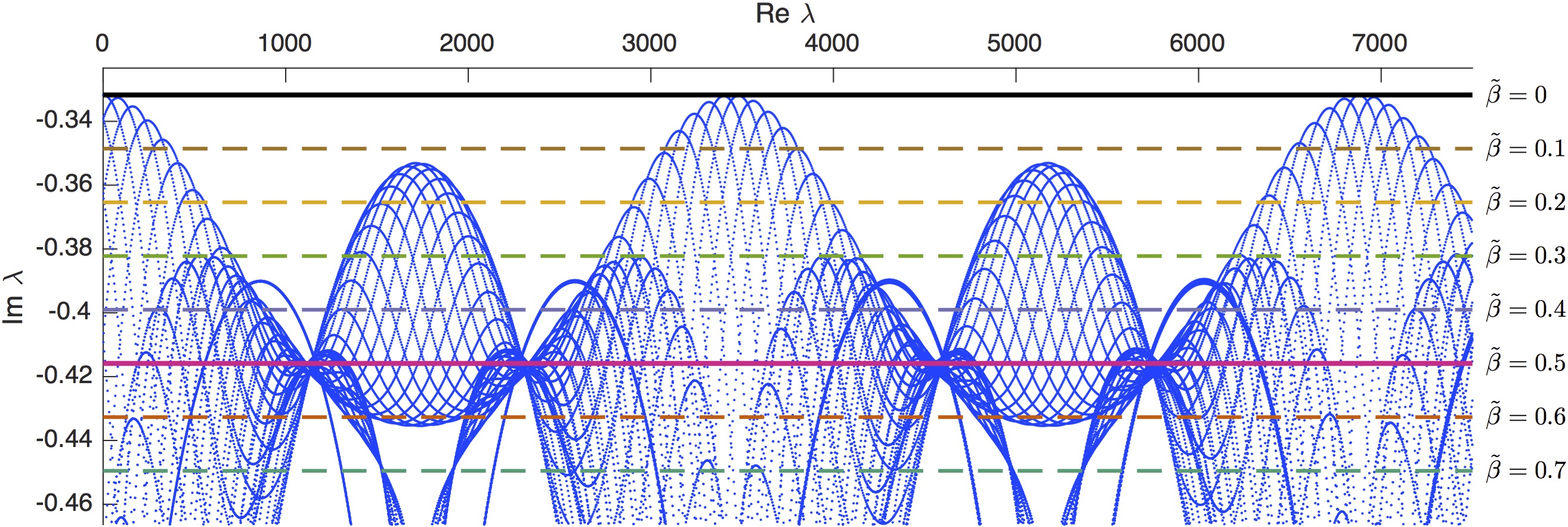}\\
\caption{\label{f:resonances}Plot of the numerically calculated resonances for the 
three-funnel surfaces $X(7,7,7)$ (top), $X(6,7,7)$ (middle) as well as the funneled
torus $Y(7,7,\pi/2)$ (bottom). The horizontal lines indicate the strips
in which the counting function is analyzed in~\S\ref{s:res_asym}.
One clearly sees the concentration phenomenon of the
resonances at $\tilde\beta=0.5$; see~\eqref{e:beta-tilde} below
for the definition of $\tilde\beta$. Additionally one sees alignment of
resonances along characteristic chains which have been studied in~\cite{WBKPS14, BFW14, We15},
as well as concentration of resonance density at $\tilde \beta=1/2$ which has been 
studied in \cite[Section 8]{BorthwickNum}.
}
\end{figure}

In~\cite{BorthwickNum} one of the authors presented an efficient numerical 
algorithm to calculate the resonances on Schottky surfaces. We refer to~\cite{BorthwickNum, Borthwick-Weich} for details and will only recall the main
steps.

The central ingredient for the numerical calculation of resonances on a convex co-compact
hyperbolic surface $M=\Gamma\backslash \mathbb H^2$ is the fact that resonances 
correspond to zeros of the Selberg zeta function $Z_\Gamma(s)$ (\ref{e:selb-zeta}). 
The series expression (\ref{e:selb-zeta}) is only absolutely convergent for 
$\mathrm{Re}(s)>\delta$, thus in the
region where no resonances are located. In the region of interest
$\mathrm{Re}(s)\leq\delta$, 
the zeta function is only given by holomorphic continuation, which is not amenable to 
numerical calculations. One can, 
however, avoid the problem of holomorphic continuation using a method introduced by
Jenkinson--Pollicott~\cite{Jenkinson-Pollicott}. These authors use dynamical zeta functions
for a transfer operator of the Bowen--Series map, which is
an expanding, holomorphic map defined using the generators~$g_i$
\begin{equation}
  \label{e:bowen-series}
\mathcal B: \bigcup_{i=1}^{2r} D_i\ \to\ \mathbb{C}\cup \{\infty\},
\end{equation}
where $D_i$ are the disks associated to the Schottky group. 
This transfer operator approach leads to a more efficient series expansion of the zeta function, which converges 
uniformly on compact sets on the full domain~$s \in \mathbb C$.

A suitable numerical approximation of the Selberg zeta function on any bounded domain 
$B\subset \mathbb C$ can be obtained by truncating the Jenkinson--Pollicott formula.
The zeros can be calculated using efficient adaptive
root finding algorithms for holomorphic functions based on the argument
principle (see e.g. the algorithm QZ-40~\cite{DSQ}). Figure~\ref{f:resonances}
shows the resulting plots of resonances in the complex plane for the surfaces $X(7,7,7)$,
$X(6,7,7)$ and $Y(7,7,\pi/2)$.

In principle, the formulas of Jenkinson--Pollicott can be used to approximate the
Selberg zeta function to arbitrary precision on any compact subset of 
$\mathbb{C}$, simply by including a sufficient number of terms in the truncated series.   
In practice, however, the complexity of the calculations increases 
exponentially as additional terms are included.
In \cite{Borthwick-Weich} two of the authors showed that a
discrete symmetry group of the surface $M$ leads to a factorization of 
$Z_\Gamma(s)$ into holomorphic symmetry-reduced zeta functions.  Using this
factorization, the numerical convergence can be dramatically improved. Still,
for practical purposes  there remain the following restrictions:  First, the calculation of resonances becomes dramatically more 
complicated for higher values of $\delta$. This effectively restricts the calculation of
resonances to surfaces with $\delta\lesssim 0.5$. Second, the calculation of 
$Z_{\Gamma}(1/2-i\lambda)$ becomes 
exponentially difficult for large negative values of $\Im \lambda$.  It is possible 
to calculate the resonances in a strip of the width of a few deltas parallel to the 
real axis, but not much beyond this. 
Third, the calculations also become exponentially complex for high 
values of $\Re\lambda$. However, here the growth of complexity is several orders
of magnitude slower compared to the case of large negative $\Im \lambda$. This allows
the computation of counting functions $\mathcal N(R,\beta)$ for the surfaces 
from Figure~\ref{f:resonances} with $\beta = 0.5-0.3\delta$ up to values of 
$R\approx 10^5$.

\subsection{Upper bounds on resonance density}

We now come to a more detailed examination of resonance densities in strips
$\{\Im\lambda\geq -\beta\}$. Let $M$ be a convex co-compact hyperbolic surface and denote
by $\Res_M$ the set of its resonances.
In order to compare results for different surfaces, it is useful to introduce a rescaled parameter 
\begin{equation}
  \label{e:beta-tilde}
\tilde \beta := \frac{\beta-1/2}{\delta} +1,\quad
\beta={1\over 2}+(\tilde\beta-1)\delta.
\end{equation}
This has the intuitive interpretation that it gives the width of the resonance counting
strip in multiples of $\delta$, with the Patterson--Sullivan gap
$\beta={1\over 2}-\delta$ corresponding to $\tilde\beta=0$ and the value 
$\tilde \beta=1/2$ corresponds to the spectral gap conjecture of \cite{Jakobson-Naud2}. Concentration of resonances near the `classical decay rate' line
$\{\Im\lambda={\delta-1\over 2}\}$ corresponding to $\tilde\beta=1/2$ was first observed (in a different setting)
in~\cite{LSZ}.

For $R,\tilde\beta>0$ we introduce the 
total counting function
\[
N(R,\tilde\beta)=\#\{\lambda\in \Res_M,\ \Re\lambda\in [0,R],\ \Im\lambda\geq -\beta\},
\]
as well as the local counting function (for fixed $L>0$)
\[
 n(R,\tilde \beta,L)=\#\{\lambda\in \Res_M,\ \Re\lambda\in [R,R+L],\ \Im\lambda\geq -\beta\}.
\]
In the case of surfaces, Theorem~\ref{t:ifwl} yields an upper bound on 
\[
n(R,\tilde\beta, 1) = \mathcal N(R,1/2 + (\tilde\beta - 1)\delta) 
\]
Note, however, that the choice $L=1$ in Theorem~\ref{t:ifwl} was made for convenience.
The same estimate applies for arbitrary fixed $L>0$, with an adjustment of the constant.   
Theorem~\ref{t:ifwl} thus implies the bounds
\begin{align}
N(R,\tilde\beta)&\leq C_{\tilde\beta} R^{1+m(\tilde\beta,\delta)+\varepsilon},\quad R\to \infty;\\
n(R,\tilde\beta,L)&\leq C_{\tilde \beta, L}R^{m(\tilde\beta,\delta)+\varepsilon},\quad R\to \infty;
\\
 \label{e:ifwl-exponent-dim2}
\text{with}\quad m(\tilde\beta,\delta)&:=\min(2\tilde \beta\delta,\delta).
\end{align}
As mentioned in the introduction, in the special case of convex co-compact 
hyperbolic surfaces Naud \cite{NaudCount} and Jakobson--Naud~\cite{Jakobson-Naud3}
previously obtained improved upper bounds on resonance densities which
we compare with the bounds of Theorem~\ref{t:ifwl}. Using the estimates
in \cite{Jakobson-Naud3}, in particular \S\S4.3,4.4 and Lemma~4.4 there,
one can derive an upper bound
\begin{align}
n(R,\tilde \beta,L)&\leq C_{\beta, L}R^{m_P(\tilde \beta,\delta)+\varepsilon},\quad R\to \infty; \nonumber
\\
\text{with }m_P(\tilde \beta,\delta)&:=\delta+\min\Big(0,{P(2\delta(1-\tilde\beta))\over \lambda_{\mathrm{max}}}\Big).
\label{e:ifwl-ecponent-press}
\end{align}
In this formula $P(x)$ is the topological pressure of the Bowen--Series map
$\mathcal B$ (see~\eqref{e:bowen-series}) and $\lambda_{\mathrm{max}}:=\max_{z\in \Lambda_\Gamma}\log(\mathcal B'(z))$
is the maximal Jacobian of $\mathcal B$ on the limit set $\Lambda_\Gamma$, 
which coincides with the maximal invariant compact set of the Bowen Series maps.

\begin{figure}
\includegraphics[scale=0.7]{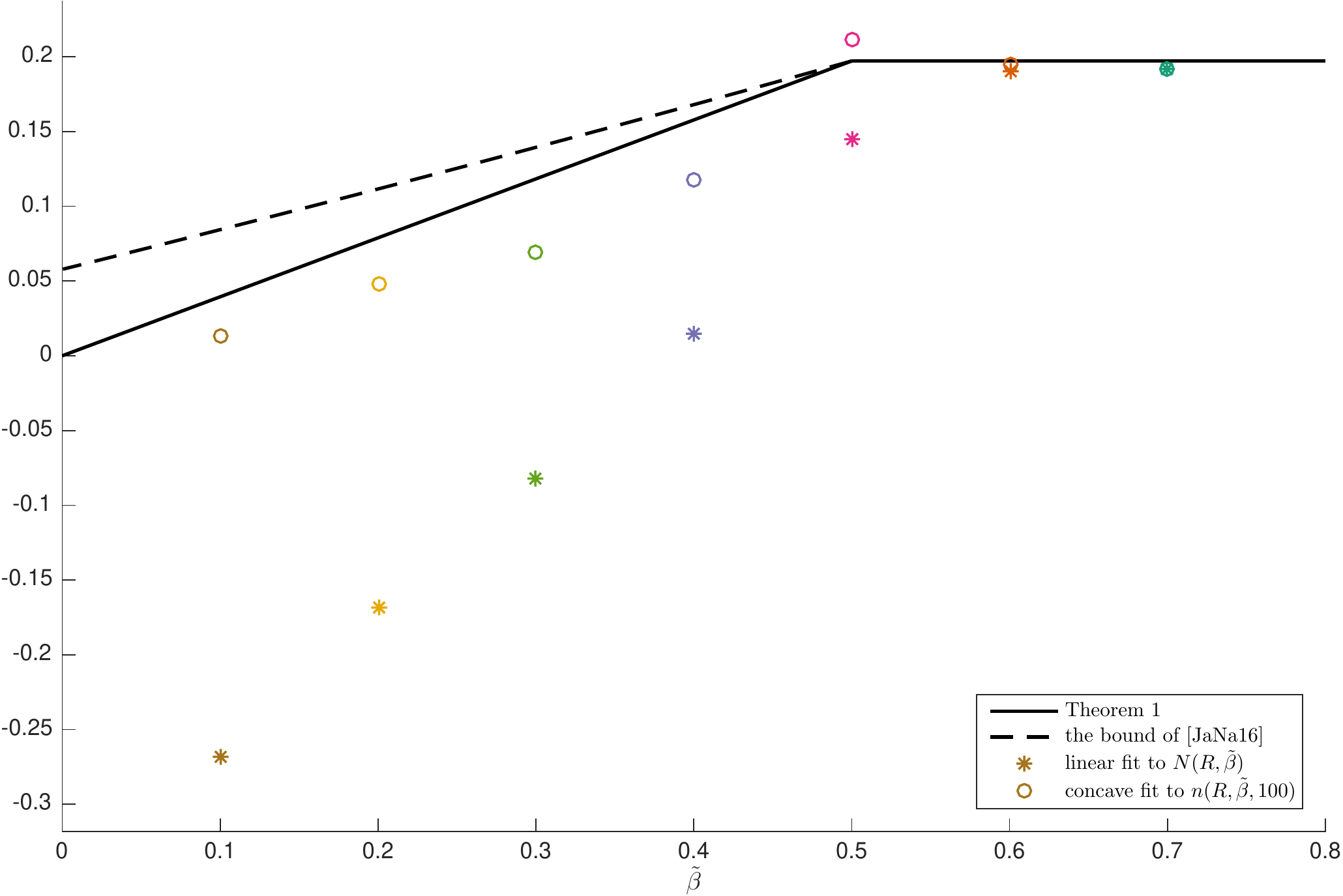}
\caption{\label{f:exponents_777} Comparison of the exponents $m_{\mathrm{fit}, N}$
(stars) and $m_{\mathrm{mean}, n}$ (circles) which have been obtained from the
numerical counting function of $X(7,7,7)$. The solid line shows the upper
bound $m(\tilde\beta,\delta)$ from Theorem~\ref{t:ifwl}, the dashed line the 
bound $m_P(\tilde\beta, \delta)$ of~\cite{Jakobson-Naud3}.
For $\tilde\beta\geq 0.5$, both of these coincide with the previous bound of~\cite{GLZ}.}
\end{figure}

In contrast to (\ref{e:ifwl-exponent-dim2}), the
bound (\ref{e:ifwl-ecponent-press}) depends crucially on the choice of the Schottky marking for a given 
convex co-compact surface. Independently of the choice of the Schottky marking one
however always has the relation
\begin{equation}
  \label{e:fn-relation}
m(\tilde \beta,\delta)\leq m_P(\tilde \beta,\delta).
\end{equation}
This can be seen as follows: For the Bowen--Series maps the topological pressure function
is continuous and monotonically decreasing, and its unique zero is given by~$\delta$.
Consequently $m(\tilde \beta,\delta)= m_P(\tilde \beta,\delta)=\delta$ for $\tilde \beta \geq 1/2$.
Additionally one knows that $P'(x) \geq -\lambda_{\mathrm{max}}$ and consequently 
\[
\frac{d}{d\tilde \beta} m_P(\tilde \beta,\delta)\ \leq\ 2\delta = \frac{d}{d\tilde \beta} m(\tilde \beta,\delta).
\]
for $0\leq\tilde \beta< 1/2$ which implies~\eqref{e:fn-relation}.

\begin{figure}
\includegraphics[scale=0.7]{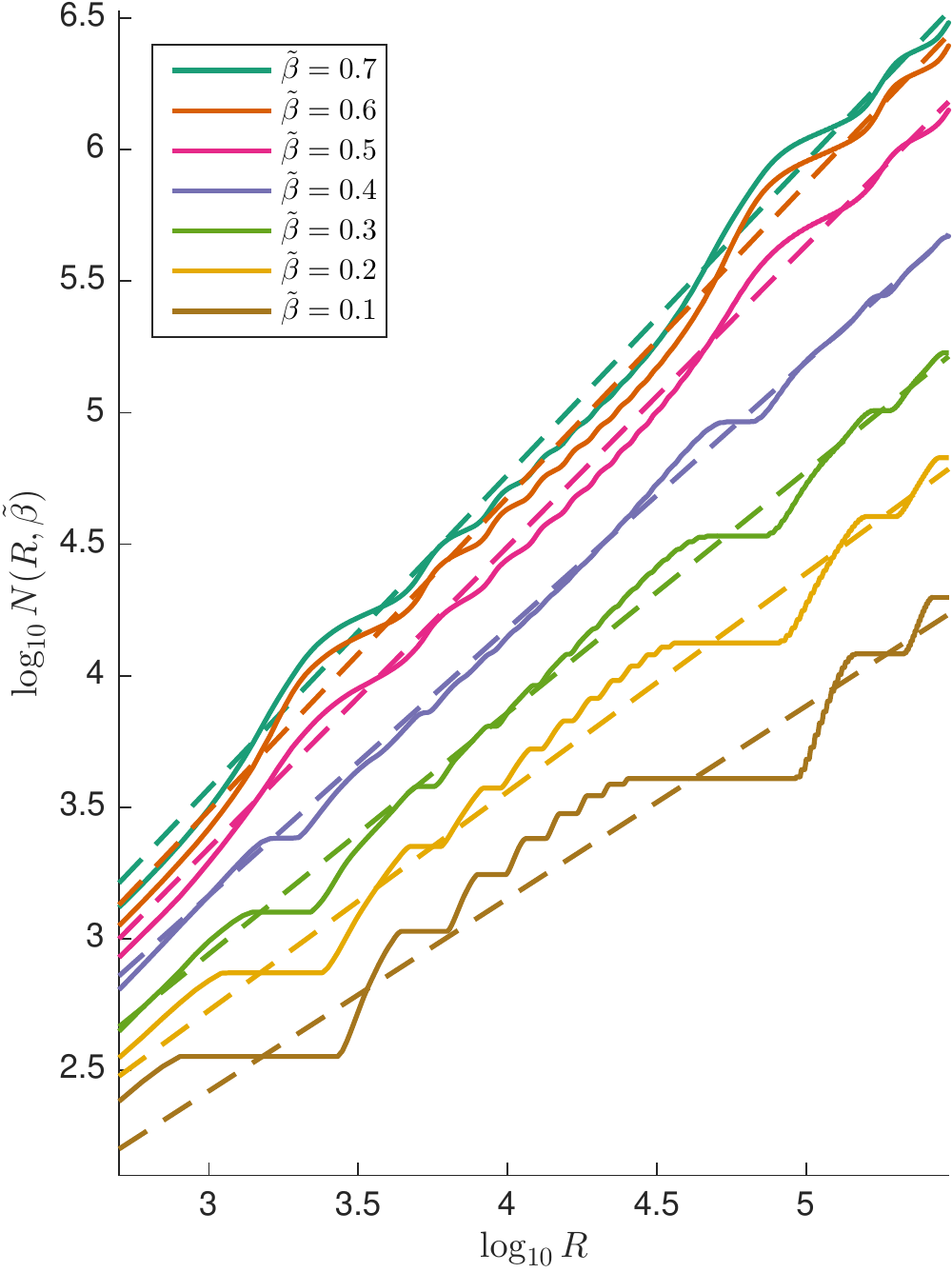}\quad
\includegraphics[scale=0.7]{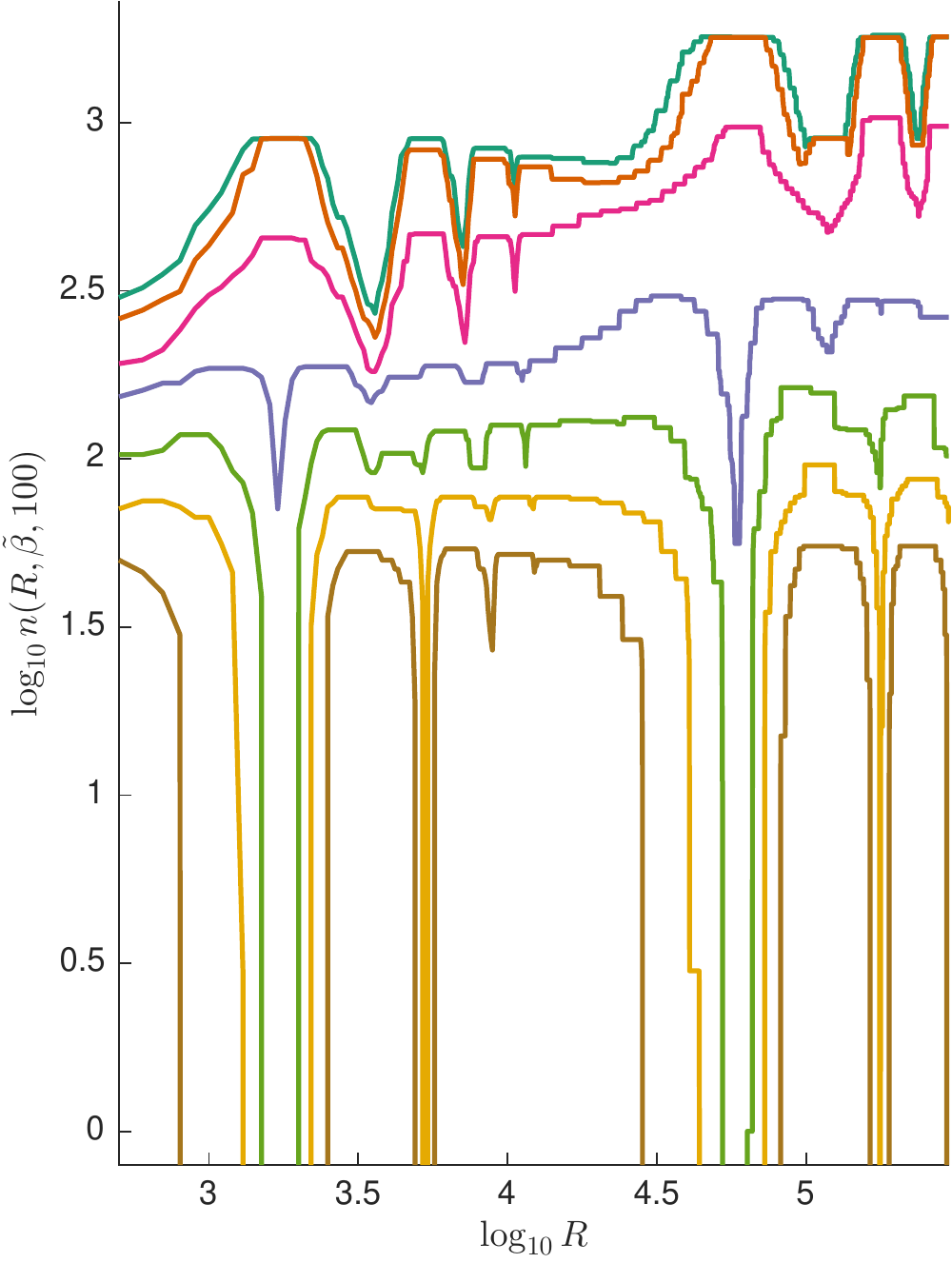}
\caption{\label{f:counting_777} Double logarithmic plot of the total counting 
function (left) and local counting function (right; see~\eqref{e:mollified}) for the three-funnel
surface $X(7,7,7)$ and different
values of $\tilde \beta$. The dashed lines in the left plot indicate linear 
fits to the double logarithmic data points, see~\eqref{e:linfits}.}
\end{figure}

In Figures~\ref{f:exponents_777}, \ref{f:exponents_677}, and \ref{f:exponents_tf77} below
we compare the two bounds for three different convex co-compact surfaces. 
For this purpose the topological pressure has been numerically calculated according to 
\cite{Jenkinson-Pollicott}. One clearly sees that in all cases 
\eqref{e:fn-relation} holds.  While the difference is rather pronounced 
for both three-funnel examples, the difference for the funneled torus is relatively small. 
The expansion rate of the Bowen--Series map for the funneled torus is much more homogeneous
than for the two other examples.   This observation therefore suggests that the two bounds 
become close to each other for surfaces that admit a very homogeneous 
Bowen--Series map.

\subsection{Comparison of theoretical upper bounds with numerics}
\label{s:res_asym}

Let us now compare the upper bounds to numerical calculations of the counting function. Using the approach described in~\S\ref{s:numerical-method}, we calculated $N(R,\tilde \beta)$ for
the surface $X(7,7,7)$ with $\tilde \beta= 0.1,0.2,\ldots, 0.7$
and $R=100, 200,\ldots, 3\cdot 10^5$.
Note that it is not necessary 
to calculate the exact position of the resonances, since the argument principle 
directly allows to calculate the number of zeros of $Z_\Gamma(s)$ in rectangular 
boxes.

A log-log plot of the total counting function is presented in the left part of 
Figure~\ref{f:counting_777}.
We observe that the counting functions behave approximately
linearly, with slopes that clearly decrease with decreasing $\tilde\beta$.  All counting functions also
show clearly visible oscillations, 
which we assume to be due to the fact that we are still in a finite-frequency 
regime. Already in the context of spectral gaps oscillations in the resonance 
pattern have been observed to be persistent up to 
very high frequencies (see~\cite[Figure~13]{Borthwick-Weich}). 
For smaller values of $\tilde \beta$, i.e. for more narrow strips, these oscillations in the
counting function become more pronounced.

We perform a linear regression to the double logarithmic data
\begin{equation}
  \label{e:linfits}
 \log(N(R,\tilde \beta)) \approx (1+m_{\mathrm{fit}, N}) \cdot \log(R) + C,
\end{equation}
where $m_{\mathrm{fit},N},C\in\mathbb R$ are chosen to minimize the sum of
squares of the difference between the left- and right-hand sides of~\eqref{e:linfits}
over all data points $R=500,600,\dots,3\cdot 10^5$.
By this we extract an exponent $m_{\mathrm{fit}, N}$ for every value of $\tilde \beta$,
and compare it to the theoretical upper bound. The parametric dependence of 
$m_{\mathrm{fit}, N}$ on~$\tilde\beta$ is shown in Figure~\ref{f:exponents_777} by the
star shaped symbols. One clearly sees that the data points for large~$\tilde \beta$
(i.e., $\tilde \beta> 0.5$) agree very well with the theoretical bound. For smaller 
values of $\tilde \beta$ the numerical values are clearly below the upper bound, 
but there are rather large deviations. In particular we obtain significantly
negative values for $m_{\textup{fit}, N}$ which implies sublinear growth of the total 
counting function. It would be interesting to understand whether this is only 
due to the restricted frequency range or a phenomenon that can be rigorousely understood.

Let us next turn to the behavior of the local counting functions. 
The right part of Figure~\ref{f:counting_777} shows a double 
logarithmic plot of $n(R,\tilde \beta,100)$ for the surface $X(7,7,7)$ and for 
different values of $\tilde \beta$. Since this function oscillates very rapidly
(see Figure~\ref{f:concave}),
we instead plot the mollified expression
\begin{equation}
  \label{e:mollified}
R\ \mapsto\ \max \{\log_{10} n(R',\tilde\beta,100)\colon |\log_{10}(R/R')|\leq 0.05\}.
\end{equation}
We have chosen $L=100$ as we want 
$L\ll R_{\mathrm{max}} = 3\cdot 10^5$ on the one hand, but on the 
other hand we want $L$ to be large relative to the resonance spacing on the chains,
which is on the order of $1$.  Once again, for different values of $\tilde \beta$  
one observes clear distinctions in the 
growth behavior of $\log(n(R,\tilde \beta,100))$. However, the most prominent
features are the strong oscillations of the local counting functions. 
In particular, for the lower values of $\tilde \beta$, i.e., for the narrower
strips, there are large $R$-ranges devoid of resonances. 

Note, however, that even an optimal asymptotic upper bound for $n(R, \tilde \beta, L)$ 
would not exclude large resonance free ranges in narrow strips along the 
real axis.  Rather it would imply that there is no better upper bound for those
frequency ranges where the resonances accumulate in the strips. It
would thus not be appropriate to extract a numerical exponent for the upper
bound by a linear fit of the double logarithmic plot. Instead, we want a 
method that extracts the mean growth rate of the regions with a high resonance
density.

We therefore chose the following two-step method for the extraction of the exponent
(see Figure~\ref{f:concave}):
\begin{itemize}
\item first, we construct the \emph{concave envelope}
$n_{\mathrm{concave}}(x)$
of the local counting function, which is the pointwise infimum of all
affine functions $x\mapsto ax+b$ which bound the
local counting function on the logarithmic scale:
$$
\max(0,\log_{10}n(R,\tilde\beta,L))\leq a\log_{10} R+b,\quad
R=500,600,\dots,3\cdot 10^5;
$$
\end{itemize}
The resulting concave envelope can be seen as the dashed line 
in Figure~\ref{f:concave}. It can be seen, that this concave 
envelope still contains boundary effects. For example, the end of 
the calculated data range happened to be in a region where 
$n(r,\tilde \beta, 100)$ takes very low values, thus the 
envelope function decays at the end of the data range. This is
obviously an artefact occuring at the boundaries of the 
finite data range. In order to get rid of these effects we perform
the
\begin{itemize}
\item second step: we define the concave envelope fit $m_{\mathrm{mean},n}$
as the slope of the straight line crossing the graph of
$n_{\mathrm{concave}}(x)$ at $x=x_1,x_2$,
where $x_1,x_2$ are the points marking $1\over 4$ and $3\over 4$
of the length in the interval $[\log_{10}(500),\log_{10}(3\cdot 10^5)]$.
\end{itemize}

\begin{figure}
\includegraphics[scale=0.7]{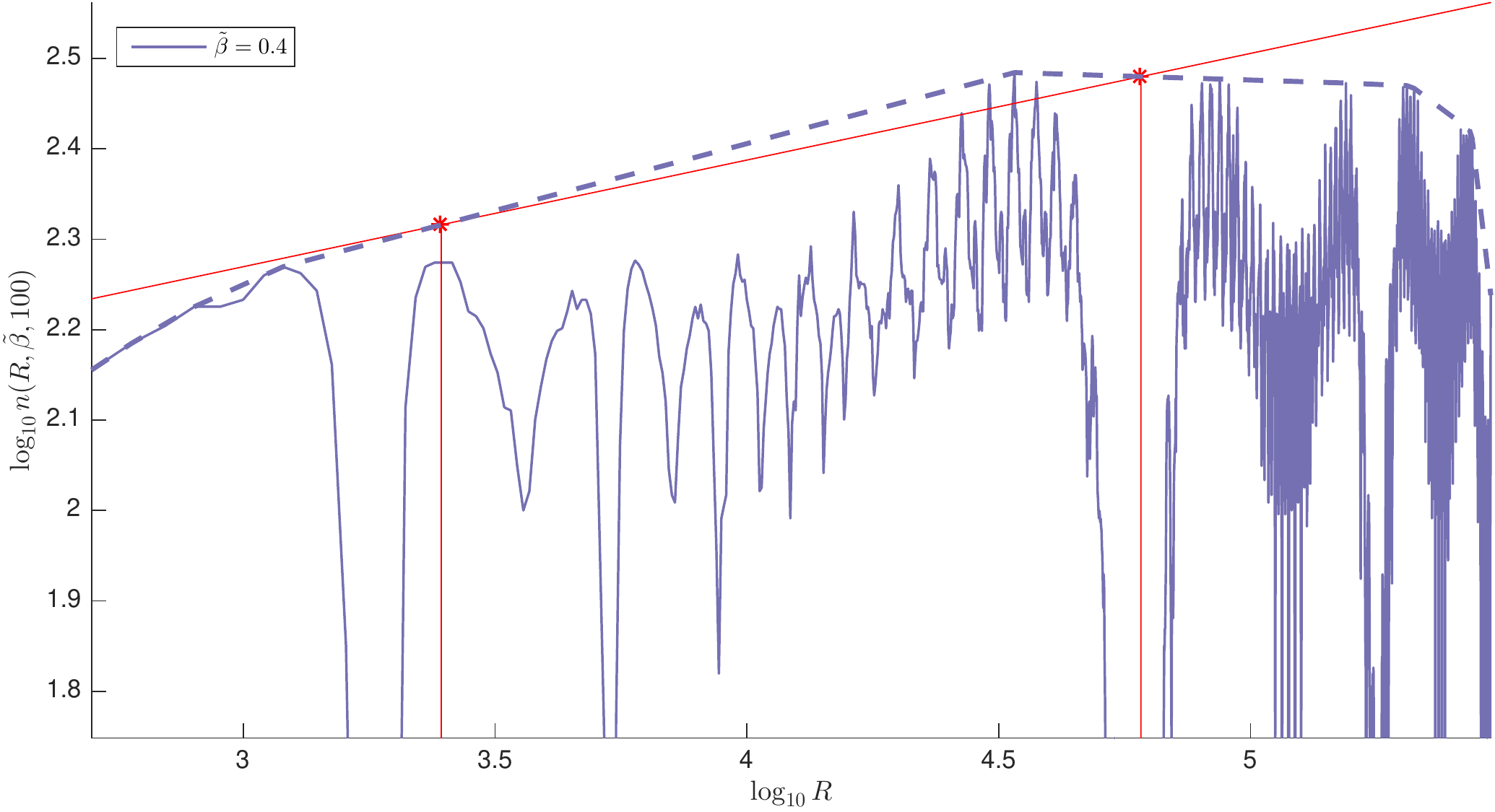}
\caption{An illustration of the concave fitting procedure
for the surface $X(7,7,7)$ and $\tilde\beta=0.4$.
The rapidly oscillating curve is the logarithmic
plot of the local counting function $n(R,\tilde\beta,100)$.
The dashed line is the concave envelope $n_{\mathrm{concave}}(x)$,
and the red line is the secant line of the concave envelope
used to determine the fit $m_{\mathrm{mean},n}$.}
\label{f:concave}
\end{figure}
The $\tilde \beta$ dependence of the quantity $m_{\mathrm{mean},n}$ is plotted 
in Figure~\ref{f:exponents_777} by the circular symbols. For large $\tilde \beta$
(i.e., $\tilde \beta> 0.5$) the exponents extracted by the total counting function
and those extracted from the local counting function agree well with each other
and also with the theoretical upper bound. For lower values of $\tilde \beta$, the exponents 
$m_{\mathrm{mean}, n}$ are significantly larger and quite close to the theoretical 
upper bounds. In view of the strong oscillations
of $n(R,\tilde \beta,L)$ this is very plausible. Fitting the log-log data of the total 
counting function to a linear function implies averaging over the oscillations
of the local counting function. The exponent $m_{\mathrm{fit}, N}$ thus also 
incorporates information of the large ranges where the local counting function
is small, whereas Theorem~\ref{t:ifwl} gives an upper bound on the 
asymptotic behavior of the maxima. 
\begin{figure}
\includegraphics[scale=0.7]{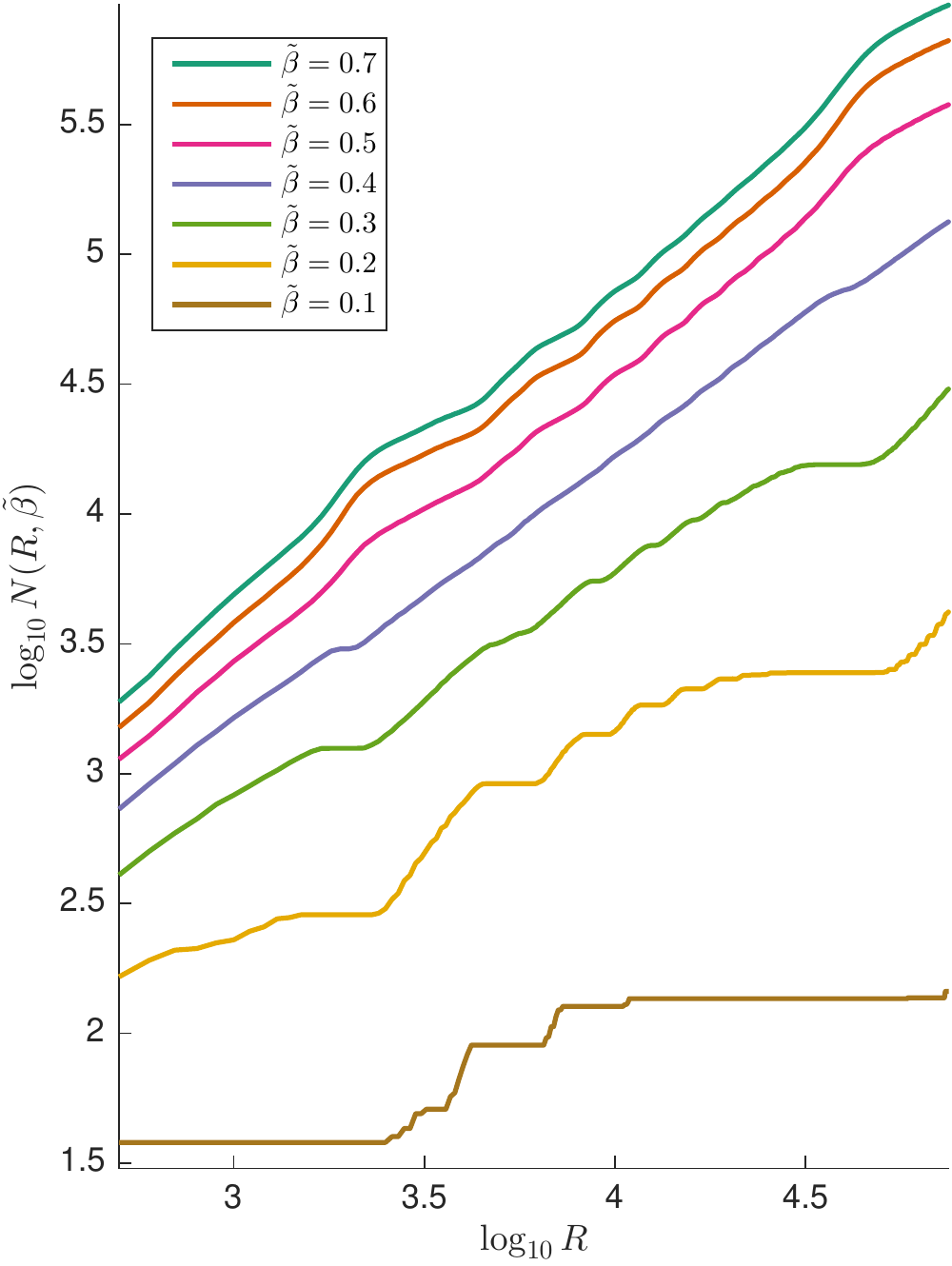}\quad
\includegraphics[scale=0.7]{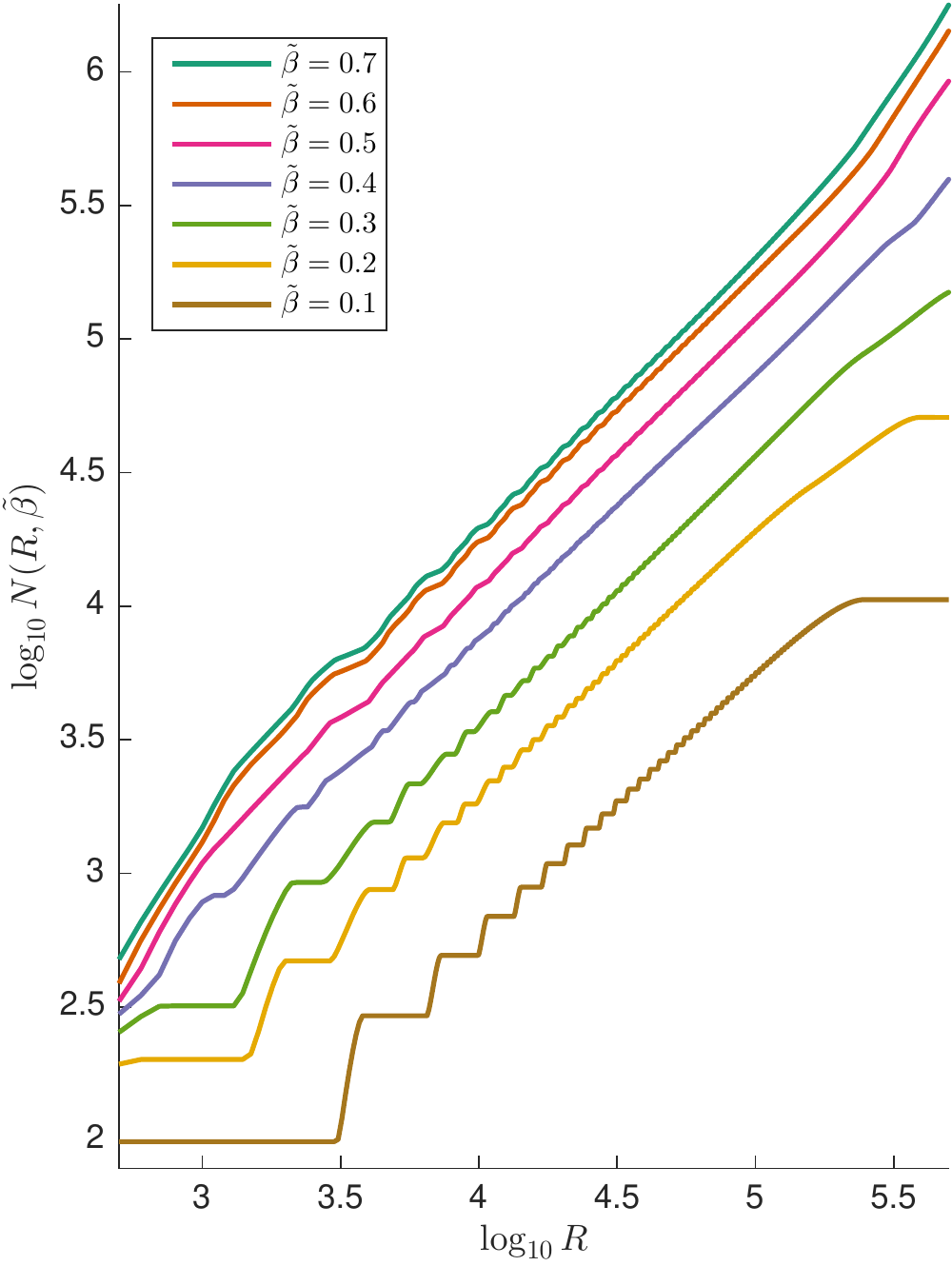}
\includegraphics[scale=0.7]{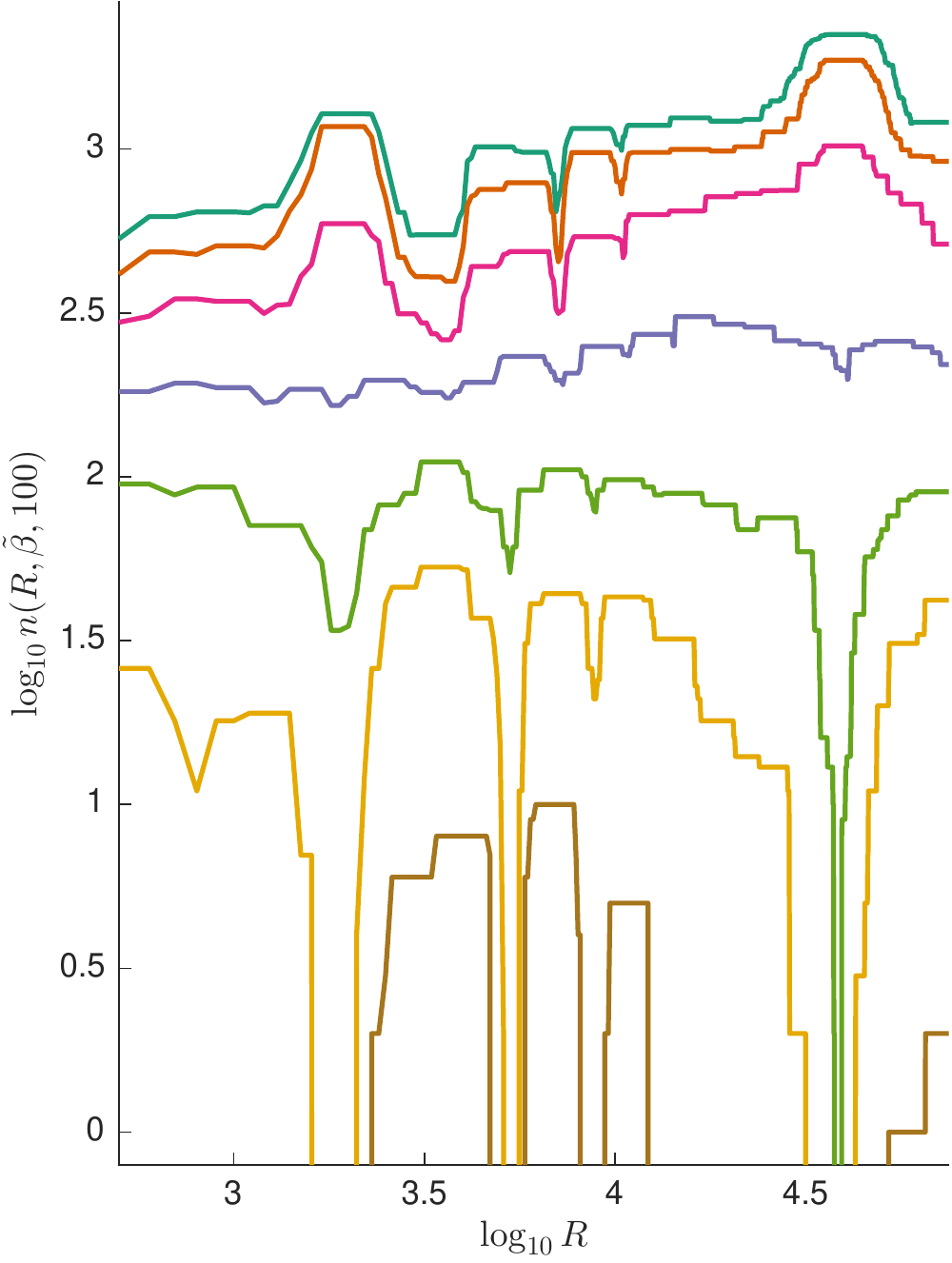}\quad
\includegraphics[scale=0.7]{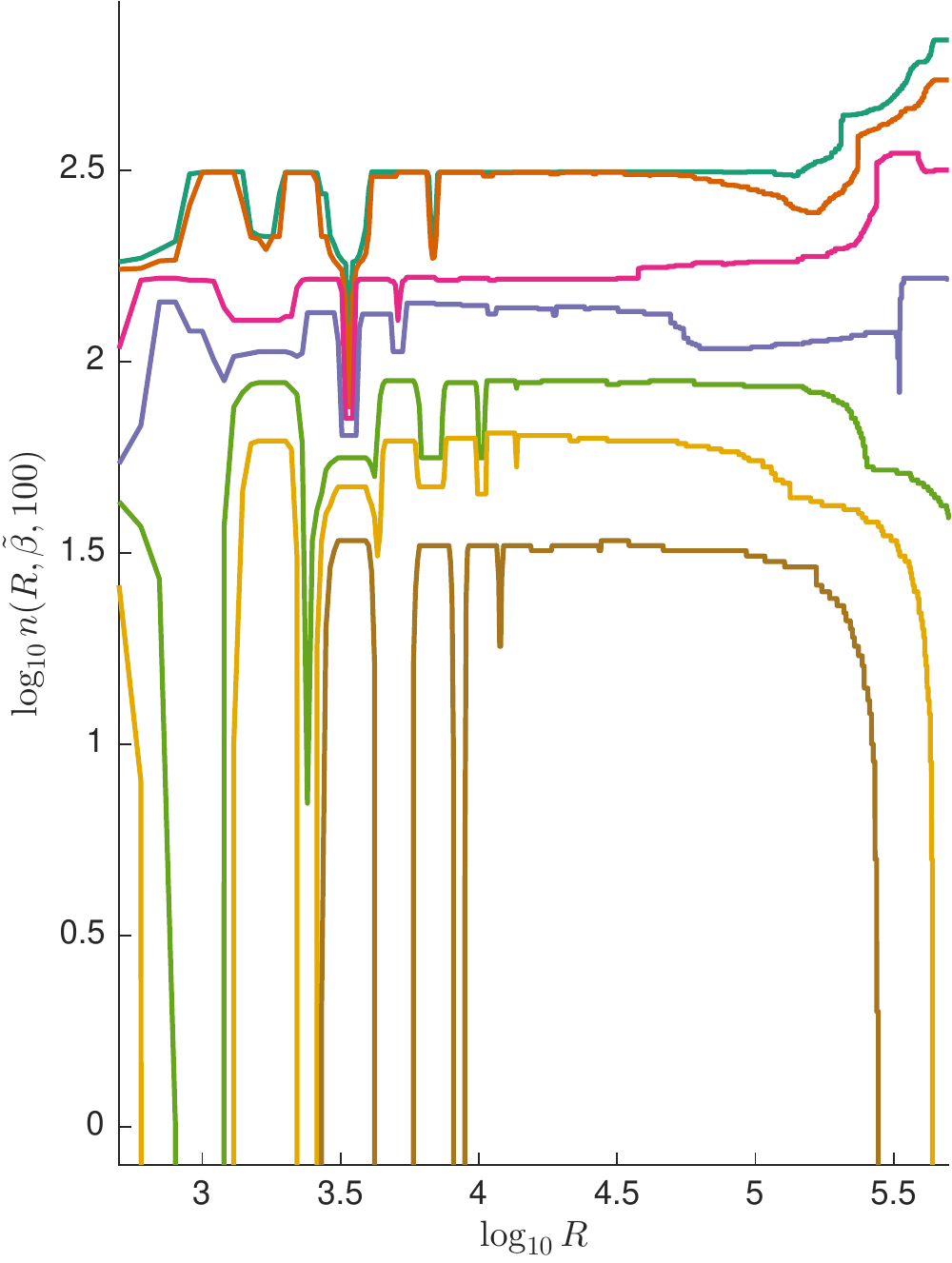}
\caption{\label{f:counting_677_tf77}Double logarithmic plot of the total counting 
function (top) and local counting function (bottom) for the three-funnel
surface $X(6,7,7)$ (left) and the funneled torus $Y(7,7,\pi/2)$ (right),
similar to Figure~\ref{f:counting_777}. Both 
data sets only represent the resonances corresponding to the trivial representation 
of the discrete symmetry groups.}
\end{figure}

Let us finally have a look at two less symmetric surfaces, the three-funnel
surface $X(6,7,7)$ and the funneled torus $Y(7,7,\pi/2)$. As these surfaces 
have a much smaller symmetry group compared to the completely symmetric surface
$X(7,7,7)$, the calculations at high frequencies are much more time-consuming.
We therefore restricted the calculation of the counting function to those resonances
that belong to the trivial representation of the discrete symmetry group 
(c.f. \cite{Borthwick-Weich}). Figure~\ref{f:counting_677_tf77}
shows double logarithmic plots of the total counting function as well as the 
local counting function. Similarly to the surface $X(7,7,7)$ both counting 
functions show oscillating behavior. In particular, for the funneled torus one sees
a visible kink in the counting function right before the end of the 
numerically accessible range, which indicates that one might need to go to 
significantly higher frequencies to see the full asymptotic behavior.  
\begin{figure}
\includegraphics[scale=0.7]{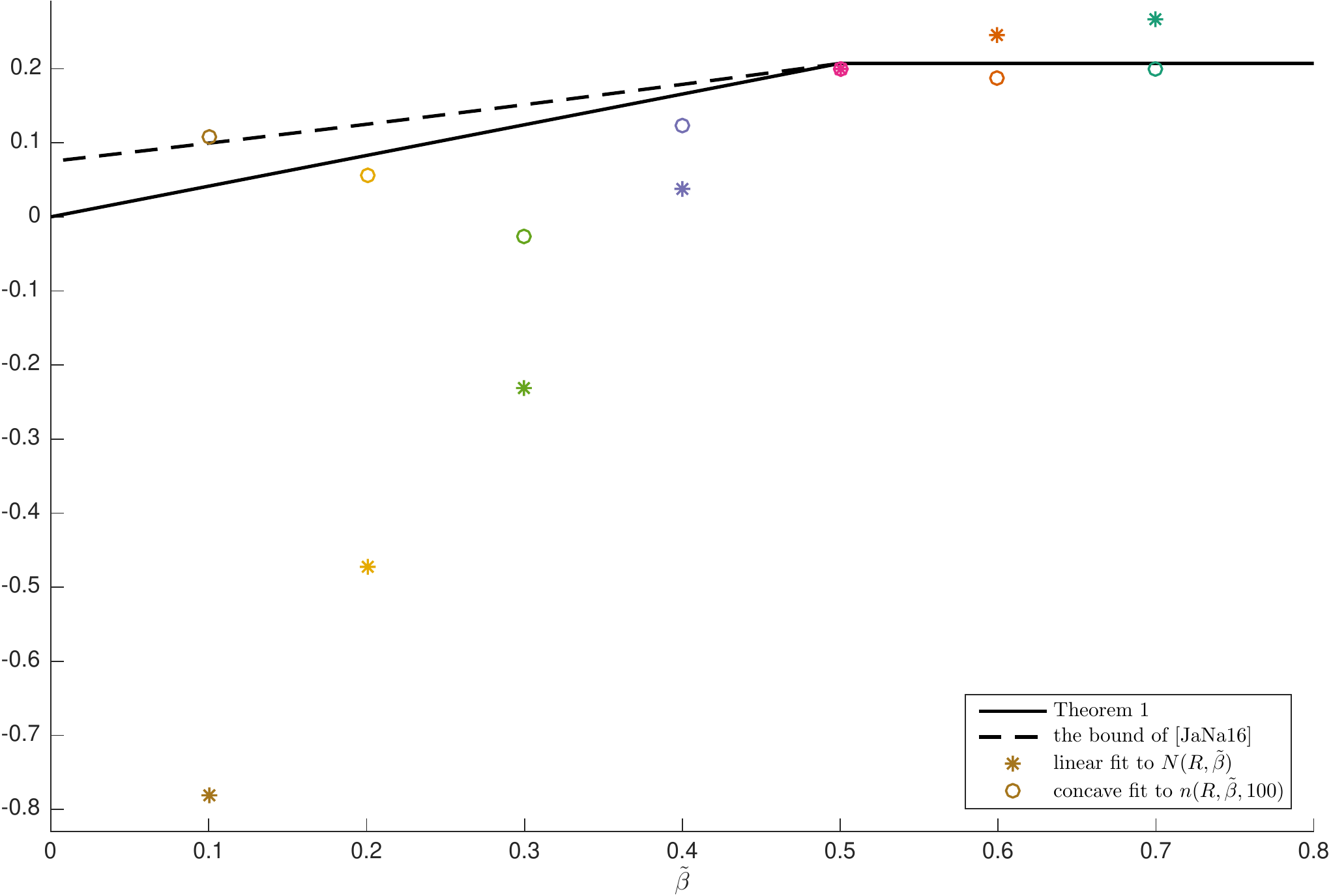}
\caption{\label{f:exponents_677} Same as Figure~\ref{f:exponents_777} but for the
surface $X(6,7,7)$.}
\includegraphics[scale=0.7]{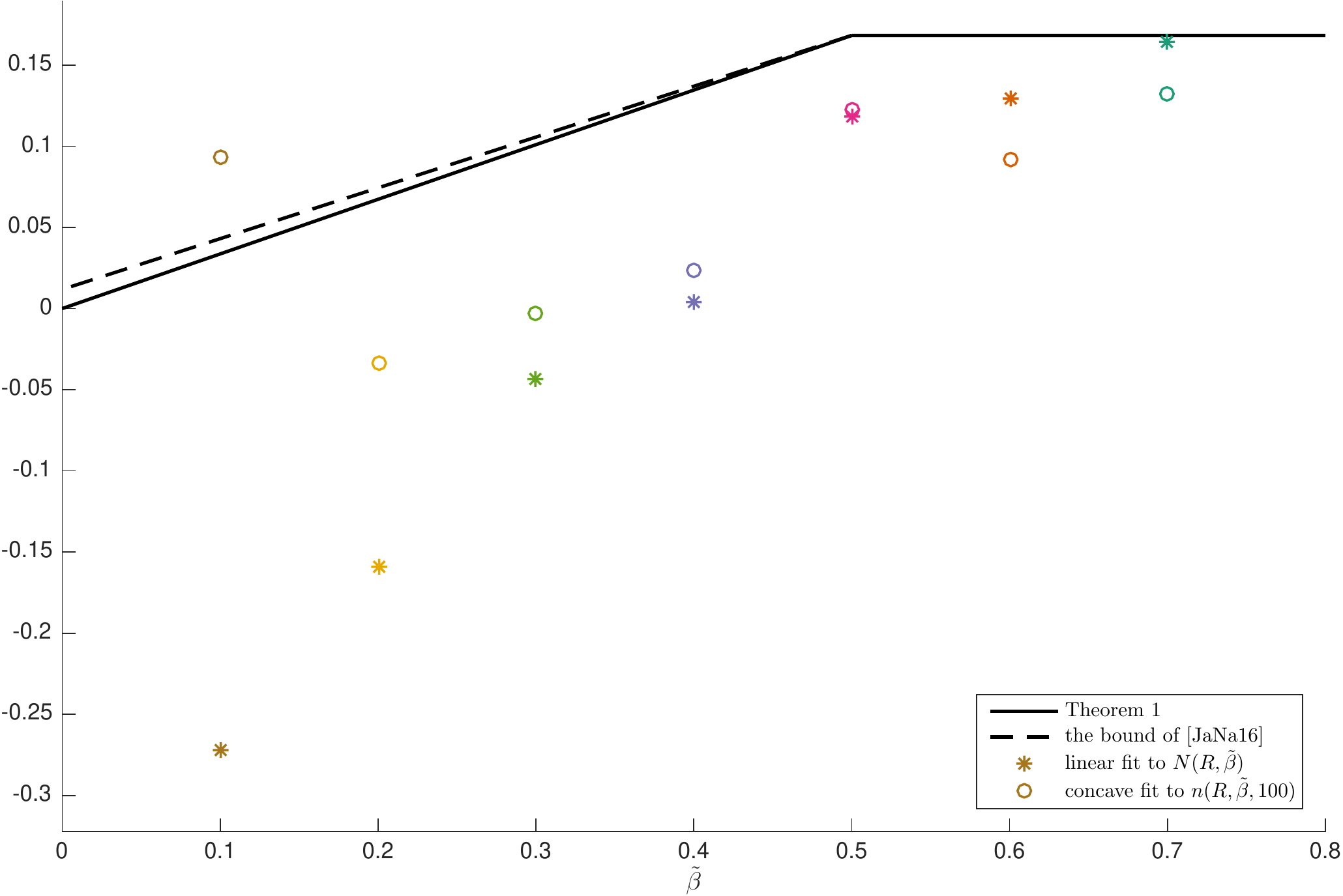}
\caption{\label{f:exponents_tf77} Same as Figure~\ref{f:exponents_777} but for the
surface $Y(7,7,\pi/2)$.}
\end{figure}

By the same procedures as above we extract the exponents $m_{\mathrm{fit}, N}$
and $m_{\mathrm{mean}, n}$ from the numerical data. The $\tilde \beta$ dependence and
a comparison with the prediction of Theorem~\ref{t:ifwl} are shown in 
Figures~\ref{f:exponents_677} and \ref{f:exponents_tf77}. Both figures
show again that the coincidence of the numerical exponent with the upper bounds 
is rather good for $\tilde \beta >0.5$. For lower $\tilde \beta$ values the exponents 
extracted from the concave upper bound are slightly below the upper bound 
of Theorem~\ref{t:ifwl}. Only for the most narrow band with $\tilde \beta = 0.1$ is
the mean exponent above this bound. However in these narrow strips there 
are huge resonance-free frequency ranges.  Thus the counting functions have 
a rather poor statistic, such that the extracted exponents have to be taken with 
caution. Comparing Figures~\ref{f:exponents_677} and \ref{f:exponents_tf77}, one
sees that the exponents for the funneled torus are much less coherent. We attribute
this to the kink described above, and assume  that the data would be more conclusive
if one could go to significantly higher frequency ranges.

In summary, we have compared the numerical data to the theoretical upper bound. 
Using the concave average method we were able to extract exponents which 
describe an asymptotic upper bound for the local counting function. The numerical 
results suggest that while the upper bound from Theorem~\ref{t:ifwl} is not 
completely optimal, it seems not to be far off for the surfaces studied.
In particular, for $X(7,7,7)$ 
(Figure~\ref{f:exponents_777}), where the high symmetry allows the most 
exhaustive numerical calculations (in particular we were able to 
calculate the spectrum of all symmetry classes) 
and which we can thus consider to be the  most reliable case, the 
exponents $m_{\mathrm{mean}, n}$ are close to the theoretical predictions.

\medskip\noindent\textbf{Acknowledgements.}
The authors would like to thank Maciej Zworski, Long Jin,
St\'ephane Nonnenmacher,
Kiril Datchev, and Colin Guillarmou
for many useful discussions regarding this project,
and Fr\'ed\'eric Naud for several discussions of~\cite{NaudCount,Jakobson-Naud3},
in particular explaining the bound~\eqref{e:ifwl-ecponent-press}.
We would also like to thank two anonymous referees for many useful comments
to improve the paper.
This research was conducted during the period SD served as
a Clay Research Fellow. TW has been supported by the 
grant DFG HI 412 12-1.


\end{document}